\documentclass[aos,reqno]{imsart}

\RequirePackage{amsmath,amsfonts,amssymb,amsthm}
\RequirePackage[numbers]{natbib}
\RequirePackage{graphicx}
\RequirePackage{mathrsfs}
\RequirePackage{dsfont}
\RequirePackage{booktabs}
\RequirePackage{nicefrac}
\RequirePackage{algorithm}
\RequirePackage{algorithmic}
\RequirePackage{url}
\RequirePackage{xcolor}
\RequirePackage[colorlinks,citecolor=blue,urlcolor=blue,linkcolor=blue,breaklinks=true]{hyperref}

\RequirePackage[nameinlink,capitalize,noabbrev]{cleveref}
\RequirePackage{autonum}

\definecolor{darkmidnightblue}{HTML}{003366}
\hypersetup{
    citecolor=darkmidnightblue,
    urlcolor=darkmidnightblue,
    linkcolor=darkmidnightblue
}

\startlocaldefs

\newcommand\bs{\boldsymbol}

\newcommand\bfS{\mathbf S}

\newcommand\bL{\boldsymbol L}

\newcommand\bY{\boldsymbol Y}

\newcommand\bX{\boldsymbol X}

\newcommand\bW{\boldsymbol W}

\newcommand\bh{\boldsymbol h}

\newcommand\by{\boldsymbol y}

\newcommand\bse{\boldsymbol e}
\newcommand\bu{\boldsymbol u}

\newcommand\bxi{\boldsymbol\xi}

\newcommand\bfI{\mathbf I}
\newcommand\bfA{\mathbf A}

\newcommand\bfX{\mathbf X}

\DeclareMathOperator{\tr}{tr}

\newcommand\bzeta{\bs \zeta}

\newcommand\btheta{\boldsymbol\theta}
\newcommand\bvartheta{\boldsymbol\vartheta}

\newcommand\bSigma{\boldsymbol\Sigma}

\def\wass{\mathsf{W}}
\def\kappav{\varkappa_{\textsf{av}}}
\def\kappamax{\varkappa_{\infty}}
\def\kapparms{\varkappa_{\textsf{rms}}}

\def\Mav{M_{\textsf{av}}}
\def\Mmax{M_{\infty}}
\def\Mtwo{M_{\mathrm{rms}}}
\def\mbar{\overline m}
\def\Mmaxbar{\overline M_{\infty}}
\def\Mavbar{\overline M_{\textsf{av}}}

\newtheorem{theorem}{Theorem}
\newtheorem{proposition}{Proposition}

\newtheorem{lemma}{Lemma}

\theoremstyle{definition}

\newtheorem{remark}{Remark}

\crefname{cor}{Corollary}{Corollaries}
\crefname{corollary}{Corollary}{Corollaries}
\crefname{condition}{Condition}{Conditions}

\endlocaldefs

\begin{document}

\begin{frontmatter}

\title{Improved Guarantees for Langevin Monte Carlo with Average Smoothness}
\runtitle{LMC: guarantees with improved condition numbers}

\begin{aug}
\author[A]{\fnms{Arnak S.}~\snm{Dalalyan}\ead[label=e1]{arnak.dalalyan@ensae.fr}}

\author[B]{\fnms{Avetik}~\snm{Karagulyan}\ead[label=e2]{avetik.karagulyan@cnrs.fr}}
\address[A]{CREST, ENSAE Paris, Institut Polytechnique de Paris and MBZUAI\printead[presep={,\ }]{e1}}
\address[B]{L2S, CNRS, Centrale Sup\'elec, Universit\'e Paris-Saclay\printead[presep={,\ }]{e2}}
\end{aug}

\begin{abstract}
We establish improved nonasymptotic bounds for Langevin 
Monte Carlo in the strongly log-concave setting, when 
the error is measured by the Wasserstein distance.  
The main result shows that the discretization error is governed by an average coordinate-wise smoothness
constant, rather than by the usual global smoothness
constant. The proof is short and probabilistic, and 
relies on a refined use of the synchronous coupling. 
We further show that the same ideas lead to improved
bounds for variable step sizes, for potentials whose
Laplacian is Lipschitz-continuous, and for finite-sum
problems sampled by stochastic-gradient Langevin dynamics
with fixed point control variates. In the Laplacian-smooth
case, the usual Hessian-Lipschitz contribution is replaced
by a weaker trace-type third-order smoothness quantity.
In the finite-sum setting, the resulting SGLD bound
improves the dependence on the root mean square smoothness
of the component functions. Applications to generalized
linear models with Gaussian design show that these
refinements can yield substantial, dimension-dependent
improvements over previously known bounds, especially for
correlated covariates.
\end{abstract}



\end{frontmatter}

\section{Introduction}

Sampling from a probability distribution with density
proportional to $e^{-f}$ is a central task in statistics,
machine learning and applied probability. In Bayesian
statistics, this problem arises when $f$ is the negative
log-posterior density; in statistical learning, $f$ often
corresponds to an empirical risk, possibly augmented by a
regularization term. In high-dimensional problems, exact
sampling is rarely available, and one must rely on
computationally tractable Markov chain Monte Carlo methods.
Among these, Langevin Monte Carlo has become one of the
most widely studied algorithms, due to its simplicity and
its explicit use of the geometry of the target distribution.

The Langevin Monte Carlo (LMC) algorithm is obtained by applying
the Euler discretization to the overdamped Langevin
diffusion. When $f$ is strongly convex and has Lipschitz
gradient, this diffusion is contractive and has the target
distribution as its invariant distribution. The core 
challenge is then to quantify the additional error introduced
by discretization. A large body of work has established
nonasymptotic guarantees for LMC in Wasserstein distance and
related metrics, under various assumptions on the potential
and the step-size sequence
\citep{DalalyanColt,dalalyan2019user,durmus2019,
durmus2019analysis}. These results have played an important
role in clarifying the dependence of LMC on the dimension,
the precision and the conditioning of the target.

A common feature of the existing guarantees is their
dependence on the global smoothness constant of the
potential. If $f$ is $m$-strongly convex and has
$\Mmax$-Lipschitz gradient, then the relevant condition
number is usually taken to be $\Mmax/m$. This worst-case
quantity is natural from the point of view of deterministic
optimization, but it can be pessimistic for sampling. Indeed,
the discretization error of LMC is partly created 
by the injection
of isotropic Gaussian noise at each step. It is therefore
reasonable to expect this error to depend not only on the
largest curvature direction, but also on an average measure
of curvature. This distinction becomes especially important
in high-dimensional problems in which the global smoothness
constant is inflated by a small number of highly curved
directions.

This paper validates this intuition. We prove that, for 
standard LMC in the
strongly log-concave setting, the leading discretization
error is governed by an average coordinate-wise smoothness
constant rather than by the global smoothness constant.
Consequently, the maximal condition number appearing in
previous guarantees can be replaced, in the discretization
term, by an average condition number. The proof is based on
a short refinement of the synchronous coupling argument. It
therefore remains close to the probabilistic intuition
behind earlier analyses of LMC, while yielding sharper
dependence on the conditioning.

\subsection{Contributions and overview}

Our first contribution is a sharper Wasserstein bound 
for LMC under strong log-concavity: we show that the 
discretization error is governed by the average 
coordinate-wise smoothness constant $\Mav$ rather 
than the global constant $\Mmax$. This refinement 
follows from a concise probabilistic argument based 
on synchronous coupling.

Our second contribution demonstrates that this 
average-smoothness perspective sharpens several 
related guarantees. For variable step sizes, it yields a
bound with improved dependence on the conditioning and
without the logarithmic dependence on the target precision.
When the Laplacian of the potential is Lipschitz-continuous,
we obtain the well-known linear-in-the-step-size 
discretization bound in which the usual Hessian-Lipschitz contribution is replaced 
by a weaker trace-type third-order smoothness quantity.

Our third contribution addresses finite-sum potentials. 
For SGLD with fixed point control variates, we obtain an
improved nonasymptotic Wasserstein bound depending on the
root mean square smoothness of the component functions.

Finally, we quantify these improvements on posteriors 
of generalized linear models with
Gaussian design. This example reveals that the improvements
can be dimension-dependent: for correlated covariates, the
gain is of order $p$ for the Lipschitz-gradient LMC bound
and of order $\sqrt p$ under the Laplacian-smoothness
assumption.

\subsection{Related work}

Langevin dynamics has its origins in the physical modeling
of Brownian motion, going back to the foundational works of
\cite{einstein1905,langevin1908} and to the
Ornstein--Uhlenbeck process \citep{uhlenbeck1930}. The 
Langevin diffusion provides a natural
continuous-time dynamics whose invariant distribution has
density proportional to $e^{-f}$. This connection between
diffusions, equilibrium distributions and stochastic
dynamics has long been central in statistical physics, and
later became a basic tool for the design of Markov chain
Monte Carlo algorithms.

The use of Langevin-type dynamics in computational
statistics developed alongside the general MCMC methodology
initiated by \cite{metropolis1953} and \cite{hastings1970}.
Gradient-based MCMC methods were promoted in statistics in
works such as \cite{grenander1994representations,
besag1994comments}, and the Metropolis-adjusted Langevin
algorithm was subsequently studied in depth. In particular,
\cite{roberts1996exponential} analyzed ergodicity properties
of Langevin algorithms and their discrete approximations,
while \cite{roberts1998optimal} studied optimal scaling
limits for Langevin algorithms. Much of this classical
literature is asymptotic in nature, focusing on invariance,
ergodicity, scaling limits and long-time behavior.

More recently, a nonasymptotic theory of Langevin Monte Carlo (LMC) has emerged. Initial contributions in this direction established Wasserstein bounds for strongly log-concave targets \cite{DalalyanColt,dalalyan2019user,durmus2019} by constructing probabilistic couplings between the continuous Langevin diffusion and its Euler discretization.
An alternative, analytical perspective interprets LMC as an optimization problem over the space of probability measures, with the Kullback--Leibler divergence serving as the objective functional \citep{cheng2018kl,wibisono2018sampling,durmus2019analysis}. This viewpoint has also yielded KL and R\'enyi divergence guarantees under isoperimetric assumptions \cite{vempala2019rapid}. Notably, \cite{durmus2019analysis} successfully leveraged this optimization-based approach to derive significantly sharper bounds for LMC than those obtained via coupling arguments.
While the present work adopts the coupling-based framework, it demonstrates that the resulting bounds can be further refined. Specifically, we improve upon previous analyses by replacing worst-case smoothness quantities with more nuanced average smoothness measures.

Several works have investigated refinements of the basic
Langevin scheme. Underdamped Langevin algorithms and kinetic variants can
improve the dependence on dimension or accuracy in some
regimes \citep{cheng2018a,dalalyan_riou_2018,
monmarche2021high,zhang2023improved}. 
Higher-order discretizations and randomized midpoint
schemes provide another route to improved guarantees under
additional or comparable smoothness assumptions
\citep{shen2019randomized,Mou21,yu2025parallel}.
Nonsmooth and constrained variants have also been studied
\citep{bubeck2018sampling,lehec2023langevin}. Proximal and 
mirror-Langevin methods have
been developed for related purposes
\citep{brosse2017proximal,durmus2018efficient, 
chewi2020exponential,ahn2021efficient,chen22c}. These works 
modify the algorithmic setting or weaken the smoothness
assumptions.
By contrast, our results sharpen the analysis of the
LMC algorithm itself by identifying more refined
smoothness quantities that govern its discretization error.

Stochastic-gradient Langevin dynamics was introduced in
\cite{WellingT11} as a scalable alternative to LMC for
Bayesian learning with large data sets. Its asymptotic
behavior with decreasing step sizes was studied in
\cite{teh2016}. For constant step sizes, the additional
noise due to stochastic gradients may create a nonnegligible
bias, motivating the use of variance reduction and control
variates \citep{BrosseDM18,Nagapetyan,Baker19}. Other
variance-reduced stochastic-gradient MCMC methods, including
SVRG- and SAGA-type Langevin algorithms, have been proposed
in \cite{dubey2016,chatterji2018,zou2019}. Our SGLD result
is closest in spirit to fixed point control-variate methods,
but gives a sharper dependence on the smoothness of the
component functions.

\subsection{Organization of the paper}

The rest of the paper is organized as follows.
\Cref{Sec:LMC} presents the main result for the standard
LMC algorithm and discusses its
consequences for mixing times and variable step sizes.
\Cref{Sec:proof} gives the probabilistic proof of
this result, based on a refinement of the synchronous
coupling argument. In \Cref{Sec:LLC}, we turn to
potentials whose Laplacian is Lipschitz-continuous and
derive a sharper discretization bound involving a
trace-type third-order smoothness quantity. The
finite-sum setting and the corresponding bounds for SGLD
with control variates are studied in
\Cref{Sec:SGLD}. In \Cref{Sec:GLM}, we evaluate the
smoothness constants appearing in the bounds for
generalized linear models, thereby
quantifying the improvements over previous guarantees.
\Cref{Sec:conclusion} concludes the paper, and the
technical proofs postponed from the main text are collected
in \Cref{Sec:proofs}.

\subsection{Notation and conventions}\label{ssec:notation}

Throughout the paper, $p$ denotes the dimension of the
state space and $\|\cdot\|_2$ denotes the Euclidean norm
on $\mathbb R^p$. For a matrix $\bfA$, we write
$\|\bfA\|_{\rm op}$ and $\|\bfA\|_{\mathsf F}$ for its
operator and Frobenius norms, respectively. We denote by
$\bse_j$ the $j$th vector of the canonical basis of
$\mathbb R^p$. For a twice differentiable function $f$,
$\nabla f$, $\nabla^2 f$ and $\Delta f = \tr(\nabla^2 f)$
stand for its gradient, Hessian and Laplacian, respectively. 
The target distribution is denoted by $\pi$ and has
density proportional to $\exp(-f)$ with respect to the
Lebesgue measure. If $\nu$ and $\mu$ are probability 
measures on $\mathbb R^p$, then $\wass_2(\nu,\mu)$ 
denotes their Wasserstein distance of order two. When 
$\bX$ and $\bY$ are random vectors, we use the shorthand 
notation $\wass_2(\bX,\bY) = \wass_2\big(\mathcal L(\bX)
,\mathcal L(\bY)\big)$. For a random vector $\bX$, we 
set $\|\bX\|_{\mathbb L_q} = \big(\mathbf E\|\bX\|_2^q 
\big){}^{1/q}$, $q\geqslant 1$. For a map $T:\mathbb R^p
\to\mathbb R^d$, its Lipschitz seminorm is  $\|T\|_{
\mathrm{Lip}} = \sup_{\btheta\neq\btheta'} \frac{\| 
T(\btheta)-T(\btheta')\|_2}{\|\btheta-\btheta'\|_2}$.
We use the notation $a\wedge b=\min(a,b)$,
$a\vee b=\max(a,b)$ and $a_+=a\vee0$. We also write
$\log_+(a)=\log(a)\vee0$. The main smoothness parameters 
are the global Lipschitz constant $\Mmax$ of $\nabla f$ 
and the average coordinate-wise smoothness constant 
$\Mav$, defined by
\begin{align}\label{coordM}
    \Mmax = \|\nabla f\|_{\rm Lip},\  
    \Mav = \frac1p\sum_{j=1}^p M_j,
    \ 
    M_j
    =\!\!\!
    \sup_{\btheta\in\mathbb R^p\!\!,\,a\in\mathbb R^*}
    \!\!\frac{2\big(f(\btheta+a\bse_j)-f(\btheta)
    -a\,\partial_j f(\btheta)\big)}{a^2}.
\end{align}
When $f$ is twice differentiable, $\Mmax$ controls the
largest eigenvalue of $\nabla^2 f(\btheta)$ uniformly in
$\btheta$, whereas $\Mav$ controls the average of the
coordinate-wise second derivatives. We write
$\kappamax = {\Mmax}/{m}$ and $\kappav = {\Mav}/{m}$ 
for the corresponding condition numbers. In the 
finite-sum setting $f=\sum_{i=1}^n f_i$, we also use
\begin{align}\label{eq:Mtwo}
    \Mtwo^2
    =
    \frac1n\sum_{i=1}^n
    \|\nabla f_i\|_{\mathrm{Lip}}^2,
    \qquad
    \kapparms
    =
    \frac{n\Mtwo}{m}.
\end{align}
Finally, when third-order smoothness is used, $M_2$ 
denotes the Lipschitz
seminorm of the Hessian with respect to the
operator norm, while $M_\Delta$ denotes the constant in
the condition $\sup_{\btheta\in\mathbb R^p} \|\nabla
\Delta f(\btheta) \|_2^2 \leqslant M_\Delta^2 p$. We set $\kappa_\Delta =
M_\Delta/m^{3/2}$ and $\kappa_2 = M_2/m^{3/2}$. 
\section{Improved upper bound for Langevin Monte-Carlo}
\label{Sec:LMC}

This section is devoted to the main result for the LMC
algorithm with exact gradient evaluations. We begin by 
stating the main theorem and outlining its proof,
postponing the proofs of several technical lemmas.  
We then  discuss the mixing time implied by the 
theorem, as well as its consequences in the case 
of varying step sizes. Finally, 
we provide proofs of the deferred lemmas. 

Langevin Monte Carlo is defined by the
sequence of random variables 
\begin{align}\label{LMC}
    \bvartheta_{k+1}^{\sf LMC} = \bvartheta_{k}^{\sf LMC} 
    - h_k \nabla f(\bvartheta_{k}^{\sf LMC})  + 
    \sqrt{2h_k}\,\bxi_k;\quad k\in\mathbb N.
\end{align}
Here, it is assumed that $(\bxi_k)_{k\in\mathbb N}$
is a sequence of independent standard Gaussian vectors,
independent of the initial condition $\bvartheta_0$, and 
that $(h_k)_{k\in\mathbb N}$ is a sequence of positive 
numbers referred to as step sizes. 

We assume that the potential function $f:\mathbb R^p\to
\mathbb R$ is strongly convex and that its gradient is
Lipschitz continuous. This is equivalent to assuming that
there exist constants $m,\Mmax\in (0,\infty)$ such that
\begin{align}\label{Lip:1}
    \frac{m}{2}\|\btheta' - \btheta\|^2\leqslant 
    f(\btheta') - f(\btheta) - \nabla f(\btheta)^\top 
    (\btheta' - \btheta)
    \leqslant \frac{\Mmax}{2}\|\btheta' - \btheta\|^2,\quad
    \forall \btheta,\btheta'\in \mathbb R^p.
\end{align}
A noteworthy feature of the analysis is that the constant
controlling the error of the LMC is not the global Lipschitz 
constant $\Mmax$, but rather the average coordinate-wise
smoothness constant $\Mav$ defined in \eqref{coordM}. Each
$M_j$ is bounded above by $\Mmax$. However, even the largest
of the $M_j$ can be significantly smaller than $\Mmax$, and this
is even more pronounced for the average smoothness $\Mav$. 
The difference between $\Mav$ and $\Mmax$ is particularly 
transparent when $f$ is twice differentiable. In this case,
$\Mmax$ provides a uniform upper bound on the largest eigenvalue 
of the Hessian matrix $\nabla^2 f(\btheta)$, whereas $\Mav$ 
controls the average of its eigenvalues, that is, the trace 
of the Hessian divided by $p$.

\subsection{The main result for the LMC with 
constant step-size}

The next theorem bounds the Wasserstein-2 error of LMC 
under the standard assumptions that $f$ is strongly convex 
and has Lipschitz gradient. We derive bounds for
both the local error, relating the error at step $k+1$ to
that at step $k$, and the global error, which controls the
error at step $k$ in terms of the initial error, the number
of steps, the step size, the strong convexity parameter,
and the smoothness parameters.

\begin{theorem}\label{th:1}
    Let $f:\mathbb R^p\to\mathbb R$ be $m$-strongly convex 
    and have $\Mmax$-Lipschitz gradient, with $0<m\leqslant 
    \Mmax$. Let $\bvartheta$ and $\bxi\sim \mathcal N(0,\mathbf 
    I_p)$ be two independent random vectors in $\mathbb R^p$. 
    If we denote by $\nu$ and $\nu^+$ the distributions of 
    $\bvartheta$ and $\bvartheta^+ =\bvartheta - h\nabla f 
    (\bvartheta) + \sqrt{2h}\,\bxi$, respectively, then for 
    every $h\geqslant 0$, we have
    \begin{align}
        \wass_2^2(\nu^+, \pi)  &\leqslant 
        \frac{1-mh}{1+mh}\,\wass_2^2(\nu, \pi) 
        + \frac{(\Mav+m)h^2p}{1+mh} + \frac{h^2(h\Mmax-1)}
        {1+mh}\|\nabla f(\bvartheta)\|_{
        \mathbb L_2}^2.\label{claim1}
    \end{align}
    The Langevin Monte Carlo algorithm with constant step-size
    $0\leqslant h\leqslant 1/\Mmax$ satisfies
    \begin{align}\label{claim3}
        \wass_2^2(\nu^{\sf LMC}_{k}, \pi) 
        \leqslant e^{-2mkh}\wass_2^2(\nu^{\sf LMC}_0, \pi) +
        \frac{(\Mav+m)hp}{2m}.
    \end{align}
\end{theorem}

The bound in \eqref{claim3} consists of two contributions:
a geometrically decaying term, which reflects the influence of
the initial distribution, and a discretization term. In squared
Wasserstein distance, the latter is of order $(\kappav+1)hp$;
equivalently, in Wasserstein distance it is of order
$\sqrt{\kappav hp}$. The first nonasymptotic Wasserstein 
bounds for the LMC algorithm 
\cite{DalalyanColt,dalalyan2019user,durmus2019}
yielded a larger discretization error, of order $\kappamax
\sqrt{hp}$. The proofs in these works are probabilistic 
and based on the synchronous coupling: one considers a
stationary Langevin diffusion  $(\bL_t)_{t\geqslant0}$, 
see \eqref{LD},  driven by a Brownian motion $(\bW_t)_{ 
t\geqslant0}$ coupled with the LMC innovations through
\begin{align}
    \sqrt{2}\big(\bW_{H_{k+1}}-\bW_{H_k}\big)
    = \sqrt{2h_k}\,\bxi_k,
    \qquad
    H_k = h_0+\ldots+h_{k-1}.    
\end{align}
The goal is then to show that $\bvartheta_k^{\textsf{LMC}}$
is close to $\bL_{H_k}$ when $k$ is large.

The discretization error was later improved to
$\sqrt{\kappamax hp}$ in \cite{durmus2019analysis}, by 
means of more advanced analytical techniques on the space 
of probability measures. While mathematically elegant, 
that approach has two drawbacks for our purposes. First, 
it does not identify a coupling under which the bound is 
achieved. Second, its complex analytical nature makes it 
difficult to adapt to other variants of the Langevin
algorithm.

In this work, we give a short probabilistic proof of
\Cref{th:1}. The proof shows that the same improved type 
of bound is obtained under the synchronous coupling. 
Its simplicity also reveals that the relevant condition 
number can be sharpened from $\kappamax$ to $\kappav$. 

\subsection{Mixing time and comparison with prior work}
\label{ssec:2.2}

Let $\varepsilon\in(0,1]$ be a prescribed precision
level, and let $\kappav = \Mav/m$ and
$\kappamax = \Mmax/m$ denote the average and maximal
condition numbers, respectively. If we require the
squared Wasserstein-2 error of the LMC algorithm to be
at most $\varepsilon^2 p/m$, and assume that the initial
error is bounded by $\varepsilon_0^2 p/m$ for some
$\varepsilon_0>0$, then it is sufficient that the number
of iterations $k$ satisfy
$mhk\geqslant \frac{1}{2}\log 
\frac{2\varepsilon_0^2} {(2\varepsilon^2 - 
(\kappav+1)mh)_+}$. This inequality and the 
constraint $h\Mmax\leqslant 1$ required in \Cref{th:1} 
are both satisfied by choosing
\begin{align}\label{eq:h}
    mh = \frac{5\varepsilon^2}{3(\kappav + 1)}
    \bigwedge \frac{1}{\kappamax}\qquad
    \text{and}\qquad
    k \geqslant \Big\{\frac{0.6(\kappav +1)}{
    \varepsilon^2} \bigvee {\kappamax}
    \Big\}\log \Bigl(\frac{\sqrt{6}\,\varepsilon_0}{
    \varepsilon}\Bigr).
\end{align}
Indeed, the choice \eqref{eq:h} ensures that
$h\Mmax\leqslant 1$ and $2\varepsilon^2 - (\kappav+1)
mh\geqslant\frac{\varepsilon^2}{3}$.

A closely related benchmark is Theorem~9 in 
\cite{durmus2019analysis}; see also Eq.~(22) therein. 
It states that
\begin{align}\label{DMM1}
    \wass_2^2(\nu^{\sf LMC}_{k}, \pi) 
    \leqslant e^{-mkh}\wass_2^2(\nu^{\sf LMC}_0, \pi)
    + \frac{2\Mmax hp}{m}.
\end{align}
The right-hand side of \eqref{DMM1} is at most 
$\varepsilon^2p/m$ if and only if $mhk\geqslant 
\log\big(\frac{\varepsilon_0^2} {(\varepsilon^2 - 
2\kappamax mh)_+}\big)$. Choosing the step size by 
the same argument as above gives $mh = \frac{5}{12}
\frac{\varepsilon^2}{\kappamax}$. Therefore, the 
constraint on $k$ becomes
\begin{align}
    k\geqslant \frac{4.8\kappamax}{\varepsilon^2}
    \log\Big(\frac{\sqrt{6}\,\varepsilon_0}{
    \varepsilon}\Big).
\end{align}
Thus, in the regime where the discretization term 
dominates, our result improves the leading constant 
by nearly a factor of eight when $\kappav=\kappamax$. 
We do not regard this constant-factor improvement 
as a main contribution; rather, it illustrates that 
the shorter proof does not come at the cost of 
degraded constants. More importantly, the bound also
replaces the maximal condition number $\kappamax$ 
by the average condition number $\kappav$ in the 
discretization term, which can lead to a significant 
improvement in ill-conditioned high-dimensional 
problems.
\begin{table}[t]
    \centering
    \begin{tabular}{cc}
    \toprule
    Our bound & Benchmark from \cite{durmus2019analysis}
    \\
    \midrule
    $\displaystyle\Big\{\frac{0.6(\kappav +1)}{
    \varepsilon^2} \bigvee {\kappamax}\Big\}\,\log 
    \Bigl(\frac{\sqrt{6}\,\varepsilon_0}{\varepsilon}
    \Bigr)$ 
    &
    $\displaystyle\frac{4.8\kappamax}{\varepsilon^2}
    \log\Big(\frac{\sqrt{6}\,\varepsilon_0}{
    \varepsilon}\Big)$
    \\
    \bottomrule
    \end{tabular}
    \caption{Comparison of the number of LMC iterations
    sufficient to guarantee a $\wass_2$-error at most
    $\varepsilon\sqrt{p/m}$, under the condition that
    $f$ satisfies strong convexity and smoothness 
    condition \eqref{Lip:1}.}
    \label{tab:lmc-mixing-comparison}
\end{table}
We close this section with a comment on the quantity
$\varepsilon_0$. Since it depends on the target
distribution $\pi$, it is not readily available in most
practical applications. However, as noted in 
\citep[Remark 2.1]{dalalyan2019user}, whenever $f$ is
nonnegative, the unknown quantity $\varepsilon_0^2$ can
be replaced by the simple upper bound $(2/p)\mathbf E 
[f(\bvartheta_0)] + 1$. 

\subsection{Bounding the error of variable step-size LMC}

The improved bound of \Cref{th:1} can also be used to
analyze the LMC algorithm with varying step sizes. The
main point is that the discretization term now depends
on the average smoothness constant $\Mav$, while the
condition ensuring that the last term in \eqref{claim1}
is non-positive remains $h_k\Mmax\leqslant 1$. As a
result, the varying step-size schedule must still be
chosen so as to respect the maximal smoothness
$\Mmax$, but the resulting error bound involves the
improved condition number $\kappav=\Mav/m$. The
proposition below makes this precise and shows that 
varying step sizes remove the logarithmic dependence 
on $1/\varepsilon$ from the mixing time.

\begin{proposition}\label{prop:VSS}
Assume that the initial distribution satisfies
$\wass_2^2(\nu_0,\pi)\leqslant \varepsilon_0^2\,
{p}/{m}$ and set $k_0:= \big\lceil
\frac{\kappamax}{2} \log_+\big( \frac{2\kappamax
\varepsilon_0^2}{\kappav+1}\big)\big\rceil$. 
Assume that the step sizes $(h_k)_{k\in\mathbb N}$ 
are chosen according to  $h_k=1/({\Mmax+m(k-k_0 
)_+})$, $k\in\mathbb N$. Then, for every $k 
\geqslant k_0$, we have
\begin{align}\label{W2VS}
    \wass_2^2(\nu_k^{\sf LMC},\pi)
    \leqslant
    \frac{\kappav+1}{\kappamax+k-k_0}\cdot\frac{p}{m}.
\end{align}
Consequently, if one requires the squared Wasserstein-2
error to be smaller than $\varepsilon^2 p/m$, it is
sufficient to choose the number of iterations $k$ so that
\begin{align}\label{kVS}
    k\geqslant k_\varepsilon:= k_0+ \Big\lceil
    \Big(\frac{\kappav+1}{\varepsilon^2}-\kappamax
    \Big)_+ \Big\rceil.
\end{align}
\end{proposition}
Variable or decreasing step sizes have already been used
in the analysis of LMC and ULA; see, for instance,
\citep{durmus2017}. Closest to the present
discussion is \citep[Theorem 2]{dalalyan2019user}, which
shows that choosing the step sizes $h_k$ of order $1/k$
removes the factor $\log(1/\varepsilon)$ from the mixing
time. Considering the initial error $\varepsilon_0$ as a
constant, the result therein gives a mixing time of order
$\kappamax^2/\varepsilon^2$, whereas \Cref{prop:VSS}
leads to the improved rate $\kappamax\log(\kappamax/
\kappav) + \kappav/\varepsilon^2$.

\section{A probabilistic proof of \Cref{th:1}}
\label{Sec:proof}

The argument relies on three auxiliary lemmas, stated when 
they are first used. Their proofs are deferred to 
\Cref{Sec:proofs}.
\begin{lemma}[Gaussian perturbation lemma]\label{lem:2}
    Let $f$ be $m$-strongly convex for some $m\geqslant 
    0$. Let $\bX$ and $\bY$ be two $p$-dimensional random 
    vectors such that $\bX$ is square-integrable and 
    $\bY$ is drawn from $\pi\propto \exp(-f)$. For any
    $p$-dimensional Brownian motion $(\bW_t)_{t\ge 0}$
    independent of\/ $\bY$, we have 
    \begin{align}
        (1+mh)\wass_2^2(\bX+\sqrt{2}\,\bW_h, \bY) &\leqslant
        \|\bX - \bY\|^2_{\mathbb L_2} + 2h
        \mathbf{E} [f(\bX+\sqrt{2}\,\bW_h) - f(\bY)] 
        \label{cor:1.0}.
    \end{align} 
    If, in addition,  $\mathbf E[\bW_t^\top \bX] = 0$ for 
    every $t\in [0,h]$,  and we set $\bX_t = \bX+\sqrt{2}\,
    \bW_t$, then 
    \begin{align}
        (1+mh)\wass_2^2(\bX_h, \bY) &\leqslant
        \|\bX - \bY\|^2_{\mathbb L_2} + 2\int_0^h 
        \mathbf{E} [f(\bX_s) - f(\bY)]\,ds + mh^2p\label{cor:1.1}.
    \end{align} 
\end{lemma}

Before proceeding with the proof, two clarifications are
in order. First, although the same notation $\bY$ is 
used on both sides of \eqref{cor:1.1}, the two 
occurrences refer to different objects in the proof. We
use the same notation only because they have the same
distribution. More precisely, the $\bY$ on the
right-hand side is optimally coupled with $\bX$,
whereas the $\bY$ on the left-hand side is the value at
time $h$ of a Langevin process started from $\bY$ and
driven by the Brownian motion $(\bW_t)_{t\geqslant 0}$.
Second, the upper bound in \eqref{cor:1.1} is sharper
than that in \eqref{cor:1.0}. This is not obvious a
priori, but it is what emerged from the proof of the
main theorem. For this reason, only \eqref{cor:1.1}
will be used in the proof of \Cref{th:1}. We
nevertheless chose to keep \eqref{cor:1.0} in the
statement of the lemma, since it holds under more
general assumptions and still provides a rather tight
bound. If one uses \eqref{cor:1.0} instead of
\eqref{cor:1.1} in the proof of \Cref{th:1}, the only
change is that the constant $\Mav + m$ is replaced by
$2\Mav$.

Let $\bY\sim \pi$ be optimally coupled with $\bvartheta$, 
so that $\mathbf E[\|\bvartheta - \bY\|_2^2] = \wass_2^2
(\nu,\pi)$. We apply \eqref{cor:1.1} to $\bX = \bvartheta 
- h\nabla f(\bvartheta)$ and a Brownian motion $\bW$ 
independent of $(\bvartheta,\bY)$. We get
\begin{align}
    (1+mh)\wass_2^2(\nu^+\!\!, \pi) 
    &= (1+mh)\wass_2^2(\bvartheta-h\nabla f(\bvartheta) + 
    \sqrt{2h}\,\bxi, \bY) \\
    &\leqslant \|\bX - \bY \|^2_{\mathbb L_2}
    +2\int_0^h\mathbf E[f(\bX_s) -f(\bY)]
    \,ds + mh^2p\\
    &\leqslant \|\bvartheta - h\nabla f(\bvartheta) - 
    \bY \|^2_{\mathbb L_2}\!
    +2\!\int_0^h\!\mathbf E[f(\bX_s) -f(\bY)]\,ds 
    + mh^2p.
    \label{eq:22}
\end{align}
We now state the second technical result providing
a tight upper bound on the first term in the last display. 

\begin{lemma}[Gradient-step lemma]\label{lem:basic-gd}
    If the function $f$ is $m$-strongly convex, then for
    every $\btheta,\by\in\mathbb R^p$ and every $h>0$, we have
    \begin{align}
        \|\btheta-h\nabla f(\btheta)-\by\|_2^2
        \leqslant
        (1-mh)\|\btheta-\by\|_2^2
        + 2h\big(f(\by)-f(\btheta)\big)
        + h^2\|\nabla f(\btheta)\|_2^2.
        \label{eq:basic-gd}
    \end{align}
\end{lemma}

Since $\btheta$ and $\by$ in this lemma are arbitrary,
we can replace them by the random variables 
$\bvartheta$ and $\bY$ and take the expectation of 
both sides of the obtained inequality. In conjunction 
with the fact that $\bX = \bvartheta - h\nabla f 
(\bvartheta)$,  this leads to 
\begin{align}
    \|\bvartheta - h\nabla f(\bvartheta) - \bY 
    \|^2_{\mathbb L_2}
    &\leqslant (1-mh) \|\bvartheta - \bY\|_{\mathbb 
    L_2}^2 + 2h\mathbf E\big[f(\bY) - f(\bvartheta)
    \big]  + h^2\|\nabla f(\bvartheta)\|_{\mathbb
    L_2}^2.
\end{align}
Combining with \eqref{eq:22}, and using  that $\bY$ is 
optimally coupled with $\bvartheta$,  we arrive at
\begin{align}
    (1+mh)\wass_2^2(\nu^+, \pi) 
    \leqslant (1-mh)\wass_2^2(\nu,\pi) &+ 2\int_0^h
    \mathbf E[f(\bX_s) -f(\bvartheta)]\,ds\\
    & + h^2\|\nabla f(\bvartheta)\|_{\mathbb 
    L_2}^2 + mh^2p. 
\end{align}
Since $\bW$ is independent of $(\bvartheta,\bY)$, 
$\bX_s =\bX + \sqrt{2}\,\bW_s$ has the same distribution 
as $\bX + \sqrt{2s}\,\bxi$, for every $s\geqslant 0$, with
$\bxi\sim \mathcal N_p(0,\bfI_p)$ independent of
$\bvartheta$. Therefore, 
\begin{align}
    (1+mh)\wass_2^2(\nu^+, \pi) 
    \leqslant (1-mh)\wass_2^2(\nu,\pi) &+ 2\int_0^h
    \mathbf E[f(\bvartheta -h\nabla f(\bvartheta) + 
    \sqrt{2s}\,\bxi) -f(\bvartheta)]\,ds\\
    & + h^2\|\nabla f(\bvartheta)\|_{\mathbb 
    L_2}^2 + mh^2p. \label{eq:20}
\end{align}
The next lemma is where the improvement from $\Mmax$ 
to $\Mav$ enters. The key idea is to exploit the 
independent coordinate structure of the Gaussian 
perturbation one coordinate at a time.

\begin{lemma}[Smoothing lemma]\label{lem:exp-upper}
    Assume that the gradient of $f$ is 
    $\Mmax$-Lipschitz and that the constants $M_j$
    and\/ $\Mav$ are defined by \eqref{coordM}. Then,
    for every  $\btheta,\btheta'\in\mathbb
    R^p$, every $h,s\geqslant 0$, and a $p$-
    dimensional random vector $\bxi$ with
    independent coordinates, each having zero mean and 
    unit variance, we have
    \begin{align}
        \mathbf E\big[f(\btheta'
        +\sqrt{2s}\, \bxi) - f(\btheta')\big]
        &\leqslant
        \Mav ps,
        \label{eq:exp-upper:a}\\
        2\big[f(\btheta-h\nabla f (\btheta)) - 
        f(\btheta)\big] &\leqslant  - h (2- 
        h{\Mmax}) \|\nabla f(\btheta) \|_2^2.
        \label{eq:exp-upper}
    \end{align}
\end{lemma}
We now combine \eqref{eq:exp-upper:a} and 
\eqref{eq:exp-upper} with \eqref{eq:20}. To this
end, we replace the vector $\btheta$ by the random
vector $\bvartheta$ and $\btheta'$ by $\bvartheta
- h\nabla f(\bvartheta)$. This substitution is
valid, thanks to the independence between $\bvartheta$ 
and $\bxi$. By dividing by $(1+mh)$, we get  
\begin{align}\label{eq:2:9}
    \wass_2^2(\nu^+, \pi) 
    &\leqslant \frac{1-mh}{1+mh}\,\wass_2^2(\nu, \pi) +
    \frac{(\Mav+m)h^2p}{1+mh} + \frac{h^2(h\Mmax -1)}{1+mh}\|\nabla f(\bvartheta)\|_{\mathbb L_2}^2.
\end{align}
This proves \eqref{claim1}, the first claim of the theorem. 

Note that when $\nu = \nu_{k}^{\sf LMC}$,  we have 
$\nu^+ = \nu_{k+1}^{\sf LMC}$. Unfolding this inequality, 
we arrive at
\begin{align}
    \wass_2^2(\nu^{\sf LMC}_{k}, \pi) 
    &\leqslant \Big(\frac{1-mh}{1+mh}\Big)^k
    \wass_2^2(\nu^{\sf LMC}_0, \pi) +
    \frac{(\Mav+m)hp}{2m} \\
    &\qquad\qquad+ \frac{h^2(h\Mmax -1)}{1+mh}\sum_{j=1}^k
    \Big(\frac{1-mh}{1+mh}\Big)^{j-1}\|\nabla 
    f(\bvartheta_{k-j}^{\sf LMC})\|_{\mathbb L_2}^2.
\end{align}
Using the inequality $(1-mh)\leqslant e^{-2mh}(1+mh)$ to
upper bound the first term of the right-hand side, and
the condition $h\Mmax\leqslant 1$ to upper bound the
last term by $0$, we get the inequality of claim 
\eqref{claim3} of the theorem.

\begin{remark}
    The probabilistic proof presented in this section has
    one limitation compared to the analytical approach of
    \cite{durmus2019analysis}. In addition to Wasserstein
    guarantees under strong convexity, the latter also
    yields, as a byproduct, upper bounds in KL-divergence
    for averaged LMC under mere convexity assumptions,
    without requiring strong convexity. We were not able
    to recover such results using the present 
    coupling-based argument.
\end{remark}

\section{Error bound when the Laplacian is 
Lipschitz-continuous}
\label{Sec:LLC}

An attractive feature of LMC, already observed in
\cite{alfonsi2015,durmus2019}, is that, when $f$ has a
Lipschitz-continuous Hessian, the dependence of the
Wasserstein-2 discretization error on the step size can
be improved from $\sqrt h$ to $h$. More precisely,
Theorem~8 in \cite{durmus2019}, see also Theorem~5 in
\cite{dalalyan2019user}, shows that if the Hessian of
$f$ is $M_2$-Lipschitz continuous, then the 
discretization error is of order
\begin{align}\label{eq:prior1}
    \big(\Mmax^{3/2} + M_2\sqrt p\big)\,\frac{h\sqrt 
    p}{m} .
\end{align}
Although this bound is quite appealing, it raises two
natural questions in light of the results of the previous
section. The first is whether the dependence on $\Mmax$
can be improved and replaced, at least partly, by the
average smoothness constant $\Mav$. The second concerns
the term $M_2\sqrt p$. This term is suboptimal for
potentials that can be written as sums of univariate
functions. Indeed, in such cases, one expects the squared
Wasserstein error to scale linearly with the dimension
$p$, whereas the presence of the factor $M_2\sqrt p$ may
lead to an upper bound of order $p^2$ for the squared
Wasserstein error. The goal of this section is to obtain
a refined version of \cite[Theorem~8]{durmus2019} and
\cite[Theorem~5]{dalalyan2019user} that addresses these
two points.

We now assume that the Laplacian of $f$ is
Lipschitz-continuous. In contrast to most previous
results, we measure third-order smoothness through a
weaker, trace-type quantity rather than through the
operator-norm Lipschitz seminorm of the Hessian. More
precisely, we assume that $f$ is three times continuously
differentiable and that
\begin{align}\label{lip:2}
    \sup_{\btheta\in\mathbb R^p} \|\nabla \Delta 
    f(\btheta)\|_2^2 \leqslant M^2_{\Delta}p.
\end{align}
This assumption is satisfied, in particular, when
$f\in C^3$ is a sum of coordinate-wise univariate
functions and its third-order partial derivatives are
uniformly bounded by $M_\Delta$. Let $M_2$ and
$M_{\mathsf F}$ denote the Lipschitz seminorms of the
Hessian when the matrix norm used in the Lipschitz
condition is, respectively, the operator norm and the
Frobenius norm. If $M_\Delta$ denotes the smallest
constant satisfying \eqref{lip:2}, then
\begin{align}
    M_\Delta\leqslant M_{\mathsf F}
    \leqslant M_2\sqrt p,
    \qquad
    M_2\leqslant M_{\mathsf F}.
\end{align}
In general, there is no ordering between $M_2$ and
$M_\Delta$, since $M_\Delta$ only controls a
trace-contraction of the third derivative tensor. As we
show in \Cref{Sec:GLM}, in high-dimensional settings,
$M_\Delta$ may be much smaller than $M_2\sqrt p$. The
next theorem shows that the previously obtained bounds
can be sharpened by replacing the terms $\Mmax^{3/2}$
and $M_2\sqrt p$ by $\Mmax\sqrt{\Mav}$ and $M_\Delta$,
respectively.

\begin{theorem}\label{th:3}
    Let $f:\mathbb R^p\to\mathbb R$ be $m$-strongly convex 
    and have $\Mmax$-Lipschitz gradient, with $0<m\leqslant 
    \Mmax$. Assume, in addition, that $f$ is three times 
    continuously differentiable on $\mathbb R^p$ and satisfies 
    \eqref{lip:2} with some constant $M_\Delta$. Then the 
    Langevin Monte Carlo algorithm with constant step-size
    $h\in (0,1/\Mmax)$ satisfies, for every $k\in\mathbb N$,
    \begin{align}\label{claim3c}
        \wass_2(\nu^{\sf LMC}_{k}, \pi) 
        \leqslant (1-mh)^k \wass_2(\nu^{\sf LMC}_0,\pi) + 
        \big(2\Mmax \sqrt{\Mav}  + M_\Delta\big)\,
        \frac{h\sqrt{p}}{m}.
    \end{align}
\end{theorem}

Comparing the discretization term in \eqref{claim3c} with
the previously known bound \eqref{eq:prior1}, we see that
the terms $\Mmax^{3/2}$ and $M_2\sqrt p$ are replaced,
respectively, by $\Mmax\sqrt{\Mav}$ and $M_\Delta$. The
first replacement improves the dependence on the smoothness
constant whenever $\Mav$ is significantly smaller than
$\Mmax$, while the second exploits the weaker trace-type
third-order smoothness condition \eqref{lip:2}. A more
quantitative assessment of these improvements is given in
\Cref{Sec:GLM}. Let us simply note here that, when $f$ is
a sum of coordinate-wise univariate functions and the
corresponding one-dimensional smoothness constants are
bounded independently of $p$, the bound in \eqref{claim3c}
has the expected $\sqrt p$ dependence on the dimension.

It is also instructive to translate \Cref{th:3} into a
mixing-time bound. Assume that $\wass_2(\nu_0^{\sf LMC}, 
\pi) \leqslant \varepsilon_0\sqrt{{p}/{m}}$, and let
$\varepsilon\in(0,1]$. Recall that $\kappa_\Delta = 
M_\Delta/m^{3/2}$ and $\kappa_2 = M_2/m^{3/2}$. 
Choosing the step size so that
\begin{align}
    mh =
    \frac{\varepsilon}{4\kappamax\sqrt{\kappav}
    +2\kappa_\Delta},
\end{align}
we have $h<1/\Mmax$, and \Cref{th:3} implies that
\begin{align}
    k \geqslant
    \frac{4\kappamax\sqrt{\kappav}
    +2\kappa_\Delta}{\varepsilon}
    \log\Big(\frac{2\varepsilon_0}{\varepsilon}\Big)
\end{align}
is sufficient to guarantee
$\wass_2(\nu_k^{\sf LMC},\pi)
\leqslant\varepsilon\sqrt{{p}/{m}}$. By contrast, 
applying the same calculation to the previously known
discretization bound \eqref{eq:prior1} would lead, up 
to universal numerical constants, to the condition
\begin{align}
    k
    \gtrsim
    \frac{\kappamax^{3/2}+\kappa_2\sqrt p}{\varepsilon}
    \log\left(\frac{\varepsilon_0}{\varepsilon}\right).
\end{align}
Thus the improvement in the discretization error directly
translates into an improved dependence of the mixing time
on the smoothness parameters.

Complementary approaches based on higher-order
discretizations or modified Langevin dynamics can lead to
improved mixing rates in some regimes, under assumptions
comparable to those imposed in this section or under 
stronger smoothness assumptions. A systematic
comparison with such refinements is beyond the scope of the
present paper, whose focus is on sharpening the analysis of
the standard LMC scheme. We refer to \cite{cheng2018a,
dalalyan2019user,dalalyan_riou_2018,Mou21} for results in
this direction.

\section{Stochastic-gradient Langevin dynamics}
\label{Sec:SGLD}

In statistics and machine learning, when $f$ is the
negative log-likelihood of a model parameterized by
$\btheta$, or the potentially penalized empirical risk
of a family of predictors parameterized by $\btheta$,
it often takes the form of a sum over the data points.
Formally, this corresponds to a potential function that
can be written as
\begin{align}\label{eq:fsum}
    f(\btheta)=\sum_{i=1}^n f_i(\btheta),
\end{align}
where $(f_i)_{1\leqslant i\leqslant n}$ is a collection
of functions from $\mathbb R^p$ to $\mathbb R$. We assume
that each $f_i$ has Lipschitz gradient and that $f$ is
strongly convex. We denote by $\mbar$, $\Mmaxbar$ and $\Mavbar$,
respectively, the strong-convexity constant, the global
smoothness constant and the average coordinate-wise
smoothness constant of the average potential $n^{-1}f$.
Equivalently, $f$ is $n\mbar$-strongly convex, has
$n\Mmaxbar$-Lipschitz gradient, and its coordinate-wise
average smoothness constant is $n\Mavbar$. In this finite-sum
setting, $\Mtwo$ is defined by \eqref{eq:Mtwo}.

Let $\btheta_*$ denote the unique minimizer of $f$, so 
that $\nabla f(\btheta_*) = 0$. We consider the 
stochastic-gradient Langevin dynamics update rule with
fixed point control variates
\begin{align}
    \bvartheta_{k+1}^{\sf SGLD} =
    \bvartheta_k^{\sf SGLD} - hn\big(\nabla f_{I_k}
    (\bvartheta_k^{\sf SGLD}) - \nabla f_{I_k}
    (\btheta_*)\big) + \sqrt{2h}\,\bxi_k,
\end{align}
where $(I_k)_{k\in\mathbb N}$ is an i.i.d.\ sequence,
uniform on $\{1,\ldots,n\}$, independent of
$(\bxi_k)_{k\in\mathbb N}$ and of the initial 
condition. As before, $(\bxi_k)_{k\in\mathbb N}$ is a
sequence of i.i.d.\ $\mathcal N(0,\bfI_p)$ random
vectors. The update above is the fixed point control-variate
version of SGLD. The estimator $n(\nabla f_{I_k}
(\btheta) -\nabla f_{I_k}(\btheta_*))$
is unbiased for $\nabla f(\btheta)$ and vanishes at
$\btheta=\btheta_*$. Its variance is therefore controlled
by the distance to the minimizer, a property that is central
to the analysis below. This distinguishes the algorithm
from versions of SGLD without the variance-reducing term,
which have been analyzed in
\cite{dalalyan2019user,durmus2019analysis}; see also
\cite{cheng20e} for related results beyond the strongly
convex setting. 

Other variance-reduced versions of stochastic-gradient
MCMC have been proposed using SVRG- and SAGA-type ideas
\citep{dubey2016,chatterji2018}. These
approaches differ from the fixed point control-variate
estimator considered here, which uses the values
$\nabla f_i(\btheta_*)$ to reduce the stochastic-gradient
noise near the minimizer. Related variance-reduced
Langevin algorithms have also been studied beyond the
strongly log-concave setting \citep{zou2019}.

\begin{theorem}\label{th:sgld}
    Let $f:\mathbb R^p\to\mathbb R$, given by 
    \eqref{eq:fsum}, be $n\mbar$-strongly convex and have 
    $n\Mmaxbar$-Lipschitz gradient, with $0<\mbar\leqslant 
    \Mmaxbar$. Let $\Mavbar$ be the average coordinate-wise 
    Lipschitz constant of $\tfrac1n\nabla f$ and $\Mtwo$ 
    be given by \eqref{eq:Mtwo}. The SGLD algorithm with 
    constant step-size
    $0\leqslant h \leqslant \frac{\mbar}{3n\Mtwo^2}$
    satisfies, for every $k\in\mathbb N$,
    \begin{align}
        \wass_2^2(\nu_k^{\sf SGLD},\pi)
        \leqslant
        (1+n\mbar h)^{-k}\wass_2^2(\nu_0^{\sf SGLD},\pi)
        + \Big\{\Mavbar+2\mbar+\frac{2\Mtwo^2}{\mbar}\Big\}
        \frac{hp}{\mbar}.
        \label{claim-sgld-3a}
    \end{align}
\end{theorem}

The conclusions of this theorem lead to an upper bound on
the query complexity of SGLD, defined as the number of
individual gradient evaluations needed to make the
Wasserstein-2 error smaller than a prescribed threshold.
It is useful to compare this query complexity with the best
available bound of the same type in the literature, and
with the query complexity of LMC obtained in \Cref{ssec:2.2}.
To this end, we consider the simplified setting in which
the initial value of SGLD is $\btheta_*$. This gives
$\wass_2^2(\nu_0^{\sf SGLD},\pi) \leqslant \frac{p}{n\mbar}$. 
In the query complexity reported below, the fixed point
control variates $\nabla f_i(\btheta_*)$ are assumed to
have been precomputed and stored; otherwise, an additional
one-time cost of order $n$ should be added.

Let $Q^{\rm SGLD}_{\rm our}$ denote the number of queries
to individual gradients $\nabla f_i$, derived from
\Cref{th:sgld}, that is sufficient for SGLD to have
Wasserstein-2 error bounded by
$\varepsilon\sqrt{p/(n\mbar)}$. Since the conditions of
\Cref{th:sgld} imply $n\mbar h\leqslant 1/3$, we have
$1+n\mbar h\geqslant 0.95 e^{n\mbar h}$. Therefore,
\begin{align}
    k \geqslant \frac{1}{n\mbar h}\log(2
    \varepsilon^{-2})
\end{align}
makes the first term in the right-hand side of
\eqref{claim-sgld-3a} smaller than
$0.6\varepsilon^2p/(n\mbar)$. Setting
$\kapparms = \Mtwo/\mbar$ and, for every
$\varepsilon\leqslant 1$, choosing $h$ so that
\begin{align}
    n\mbar h\leqslant \frac{\varepsilon^2}{5 +
    2.5\kappav + 5\kapparms^2},
\end{align}
we get the second term in the right-hand side of
\eqref{claim-sgld-3a} smaller than
$0.4\varepsilon^2 p/(n\mbar)$. Thus, one can infer from
\Cref{th:sgld} that
\begin{align}
    Q^{\rm SGLD}_{\rm our}
    \leqslant
    \frac{5 + 2.5\kappav + 5\kapparms^2}
    {\varepsilon^2}\,
    \log(2\varepsilon^{-2}).
\end{align}
Note that the dominant term in the right-hand side is
$\kapparms^2$, since
$1\leqslant \kappav\leqslant\kappamax\leqslant
\kapparms$.

For the fixed point control-variate SGLD considered here,
a directly comparable bound in the prior literature can be
derived from \cite[Theorem 2]{BrosseDM18}. Their analysis
requires additional assumptions, involving third- and
fourth-order derivatives and convexity of each $f_i$, 
which are not needed in \Cref{th:sgld}. Their result yields
\begin{align}
    Q^{\rm SGLD}_{\rm prior}
    \leqslant
    \textsf{C}\,
    \frac{\kapparms^3}{\varepsilon^2}
    \log(\varepsilon^{-1})
\end{align}
for some universal constant $\textsf{C}$. Compared 
to our result, this contains an extra factor
$\kapparms$, which can be large in ill-conditioned
problems; see \Cref{Sec:GLM} for further discussion.
For comparison, \cite{chatterji2018} also studied
variance-reduced stochastic-gradient MCMC algorithms,
including SAGA-LD, SVRG-LD and a control-variate
underdamped Langevin algorithm. For the latter method,
which may be viewed as an underdamped counterpart of the
control-variate SGLD considered here, their bound gives
$Q^{\sf SGULD}_{\rm prior}\leqslant
\textsf{C}\,\frac{\kapparms^{5.5}}{\varepsilon^3}
\log(\varepsilon^{-1})$. For ease of comparison, we 
report these bounds in
\Cref{table:sgld-comparison}. We do not include the
SVRG- and SAGA-type Langevin algorithms studied in
\cite{chatterji2018}, because their complexity bounds
are stated for an absolute Wasserstein accuracy.
Translating these bounds to the present accuracy level
$\varepsilon\sqrt{p/(n\mbar)}$ introduces additional factors
depending on $n$, and does not improve the comparison in
the large-sample regimes considered here.

Let us now compare the query complexity of SGLD with that
of LMC. Each iteration of LMC requires the computation of
all $n$ gradients $\nabla f_i$. Hence its query complexity
is bounded by $n$ times the number of steps needed to
achieve an error smaller than
$\varepsilon\sqrt{p/(n\mbar)}$. This gives the value reported
in \Cref{table:sgld-comparison}: LMC has a better
dependence on the condition numbers, linear in $\kappav$
and $\kappamax$ instead of quadratic in $\kapparms$, but a
significantly worse dependence on $n$. This confirms and
quantifies the fact that, in statistical problems with a
very large sample size, SGLD can outperform LMC.

\begin{table}[t]
    \centering
    \begin{tabular}{@{}ccc|c@{}}
    \toprule
    Our SGLD &  SGLD \cite{BrosseDM18} & SGULD 
    \cite{chatterji2018} & Our LMC
    \\
    \midrule
    $\displaystyle
    \frac{\kapparms^2}
    {\varepsilon^2}\log(\varepsilon^{-1})$
    &
    $\displaystyle
    \frac{\kapparms^3}{\varepsilon^2}
    \log(\varepsilon^{-1})$
    &
    $\displaystyle
    \frac{\kapparms^{5.5}}{\varepsilon^3}\log(
    \varepsilon^{-1})$
    &
    $\displaystyle
    n\Big\{\frac{\kappav}{\varepsilon^2}
    \bigvee \kappamax \Big\}\log(\varepsilon^{-1})$\\
    \bottomrule
    \end{tabular}
    \caption{Comparison of orders of query complexities 
    sufficient to guarantee a $\wass_2$-error at most
    $\varepsilon\sqrt{p/(n\mbar)}$. For SGLD, one query
    corresponds to one individual gradient evaluation,
    after precomputation of the fixed point control
    variates. For LMC, one iteration requires $n$ such
    evaluations.}
    \label{table:sgld-comparison}
\end{table}

\section{Application to Posteriors of Generalized Linear 
Models}\label{Sec:GLM}

A particularly attractive feature of the bounds 
established in the previous sections is their improved 
dependence on the relevant condition numbers. To 
quantify the impact of this improvement, we consider 
here the concrete example of sampling from a density 
whose potential is of the form
\begin{align}\label{GLM}
    f(\btheta) = \sum_{i=1}^n g_i(\bX_i^\top\btheta),
\end{align}
where $g_i:\mathbb R\to\mathbb R$ are convex functions, 
assumed to be three times continuously differentiable, 
and $\bX_1,\ldots,\bX_n$ are independent random vectors 
drawn from the Gaussian distribution $\mathcal N_p(0, 
\bSigma)$. For simplicity, we assume $\bSigma$ to be 
invertible. 
This corresponds to the posterior distribution of a
generalized linear model \citep{gelman2013,ibrahim1991,
nelder1972}, either with an improper flat prior or 
with a Gaussian prior whose covariance matrix is 
proportional to $(\sum_{i=1}^n \bX_i\bX_i^\top)^{-1}$. 
The dependence on the responses, if present, is 
absorbed into the functions $g_i$. The potential can 
also be seen as the empirical loss associated with a 
random-features model, or equivalently with a 
two-layer neural network in which the hidden-layer
weights are fixed and only the output-layer weights 
are trained \citep{bach2017,gerace20,mei2022}. 
This example is particularly relevant for the present 
paper, since it allows for a rather precise 
quantitative assessment of the constants associated 
with the full potential $f$, namely $\Mmax$, $\Mav$, 
$\Mtwo$, $M_\Delta$ and $M_2$.

\begin{proposition}\label{prop:glm-constants}
    Let the potential $f$ be given by \eqref{GLM}, 
    where the functions $g_i:\mathbb R\to\mathbb R$ 
    are three times continuously differentiable, 
    and satisfy
    \begin{align}
        \sup_{1\leqslant i\leqslant n}
        \sup_{t\in\mathbb R}(|g_i''(t)| + |g_i'''
        (t)|)\leqslant  C_g .
    \end{align}
    Let $(\bX_i)_{1\leqslant i\leqslant n}$ be independent
    $\mathcal N_p(0,\bSigma)$ random vectors with
    $p\leqslant n$. Then, with probability at least $1/2$,
    the constants $\Mmax$, $\Mav$, $M_2$, $M_\Delta$ and
    $\Mtwo$, all associated with the full potential $f$,
    are bounded from above, up to multiplicative constants
    depending only on $C_g$, by the quantities displayed
    in the table:
    \begin{center}
    \begin{tabular}{@{}c|cc|cc|c@{}}
        \toprule
        $\bSigma$
        & $\Mmax$
        & $\Mav$
        & $M_2\sqrt{p}$
        & $M_\Delta$
        & $\Mtwo$
        \\
        \midrule
        $\mathbf I_p$
        & $n$
        & $n$
        & $n\sqrt p$
        & $n\sqrt p$
        & $p$
        \\
        $\mathbf I_p+\mathbf 1_p\mathbf 1_p^\top$
        & $np$
        & $n$
        & $np^2$
        & $np$
        & $p$
        \\
        \bottomrule
    \end{tabular}
    \end{center}
    Here, $\Mmax$ is the global Lipschitz constant of
    $\nabla f$, while $\Mav$ is the average of the
    coordinate-wise Lipschitz constants defined in
    \eqref{coordM}. The constant $M_2$ is the Lipschitz
    constant of $\nabla^2 f$ when matrices are equipped
    with the operator norm. The constant $M_\Delta$ is
    defined by \eqref{lip:2}, and $\Mtwo$ is the root mean
    square of the Lipschitz seminorms of $\nabla f_i$ with
    $f_i(\btheta)=g_i(\bX_i^\top\btheta)$.
\end{proposition}

The entries of the table in \Cref{prop:glm-constants} are 
also sharp in general. Indeed, standard non-asymptotic random 
matrix estimates for sub-Gaussian matrices
\citep{vershynin2018high} imply that, under a suitable 
nondegeneracy condition on the functions $g_i$, ensuring 
in particular that their second- and third-order derivatives 
do not vanish on sets of positive measure, all entries of 
the table are also lower bounds, up to multiplicative 
constants, and on an event of positive probability, for the 
corresponding quantities $\Mmax$, $\Mav$, $M_2$, $M_\Delta$ 
and $\Mtwo$.

This proposition shows that, for uncorrelated covariates,
that is, when $\bSigma$ is the identity matrix, replacing
$\Mmax$ by $\Mav$ and $M_2\sqrt p$ by $M_\Delta$ does not
improve the order of the bounds. In contrast, for 
correlated high-dimensional covariates, as illustrated 
by the covariance matrix $\bfI_p+\mathbf 1_p\mathbf 
1_p^\top$, these replacements lead to a substantial 
improvement: a factor of order $p$ in the 
Lipschitz-gradient LMC bound, and a factor of order 
$\sqrt p$ in the Hessian-smooth LMC bound.

To translate these estimates into the complexity bounds
proved in the previous sections, let $\mbar$ denote the
strong-convexity constant of the average potential
$n^{-1}f$. We impose the normalization $\mbar\asymp 1$.
This is the natural scaling in the present Gaussian-design
model; for instance, it holds with positive probability,
and in fact with high probability under standard
aspect-ratio assumptions, if the functions $g_i''$ are
uniformly bounded away from zero. Under this normalization,
the strong-convexity constant of the full potential $f$ is
of order $n$, and the target accuracy in the following
complexity comparisons is $\varepsilon\sqrt{p/n}$.

We report query complexities, where one query means one
evaluation of an individual gradient $\nabla f_i$. Thus
one iteration of LMC costs $n$ queries. Logarithmic
factors in $\varepsilon^{-1}$ and numerical constants are
suppressed. For compactness, write
$\bSigma_{\rm corr} = \bfI_p+\mathbf 1_p\mathbf 1_p^\top$.
The resulting orders are summarized in
\Cref{tab:glm-complexities}.
\begin{table}[t]
    \centering
    \begin{tabular}{@{}c|cccc|cc@{}}
    \toprule
    $\bSigma$
    & \shortstack{Prior LMC \cite{durmus2019analysis}
    \\Lip. grad.}
    & \shortstack{Our LMC\\\Cref{th:1}}
    & \shortstack{Prior LMC \cite{durmus2019,dalalyan2019user}
    \\Lip. Hess.}
    & \shortstack{Our LMC\\\Cref{th:3}}
    & \shortstack{Our SGLD\\\Cref{th:sgld}}
    & \shortstack{Prior SGLD\\\cite{BrosseDM18}}
    \\
    \midrule
    $\bfI_p$
    & $\displaystyle n\varepsilon^{-2}$
    & $\displaystyle n\varepsilon^{-2}$
    & $\displaystyle n\varepsilon^{-1}$
    & $\displaystyle n\varepsilon^{-1}$
    & $\displaystyle p^2\varepsilon^{-2}$
    & $\displaystyle p^3\varepsilon^{-2}$
    \\
    $\bSigma_{\rm corr}$
    & $\displaystyle np\,\varepsilon^{-2}$
    & $\displaystyle n\big(\varepsilon^{-2}\vee p\big)$
    & $\displaystyle np^{3/2}\varepsilon^{-1}$
    & $\displaystyle np\,\varepsilon^{-1}$
    & $\displaystyle p^2\varepsilon^{-2}$
    & $\displaystyle p^3\varepsilon^{-2}$
    \\
    \bottomrule
    \end{tabular}
    \caption{Orders of query complexities in the
    Gaussian-design GLM example for target accuracy
    $\varepsilon\sqrt{p/n}$ in $\wass_2$. Logarithmic 
    factors and the SGLD precomputation cost
    are omitted.}
    \label{tab:glm-complexities}
\end{table} 
In the isotropic case, the new LMC bounds have the same
order as the previously known ones. In contrast, for the
correlated design $\bSigma=\bSigma_{\rm corr}$, the
improvement is substantial. For the Lipschitz-gradient
bound, the dependence on the condition number is reduced
from $p\varepsilon^{-2}$ to $\varepsilon^{-2}\vee p$, that
is by a factor that is often of the order of dimension $p$.
For the Hessian-smooth bound, the improvement comes from
replacing both $\Mmax^{3/2}$ by $\Mmax\sqrt{\Mav}$ and
$M_2\sqrt p$ by $M_\Delta$, leading to an improvement by
a factor of order $\sqrt p$ in the displayed regime.

The same example also illustrates the improvement obtained
for the fixed point control-variate SGLD. Assuming that the
normalization $\mbar =1$ is applied, the condition number
$\kapparms=\Mtwo/\mbar$ is of order $p$ for both
covariance structures. Therefore, up to logarithmic
factors, \Cref{th:sgld} gives a query complexity of order
$p^2\varepsilon^{-2}$, whereas the directly comparable
prior SGLD bound scales as $p^3\varepsilon^{-2}$.

\section{Conclusion}
\label{Sec:conclusion}

In this paper, we revisited nonasymptotic Wasserstein
bounds for Langevin Monte Carlo in the strongly
log-concave setting. We showed that the leading
discretization error of the standard LMC algorithm is
controlled by the average coordinate-wise smoothness of
the potential, not by the global smoothness constant. As
a consequence, the maximal condition number appearing in
previous guarantees can be replaced, in the discretization
term, by an average condition number. This refinement is
particularly significant in high-dimensional problems in
which the global smoothness is inflated by a small number
of highly curved directions. The proof is based on a short
refinement of the synchronous coupling argument, and
therefore remains close to the probabilistic intuition
behind earlier analyses of LMC.

Motivated by this observation, we revisited the dependence
on smoothness and condition numbers in several closely
related guarantees. The guiding principle is that the
condition numbers appearing in earlier bounds may be
overly pessimistic because they are built from global
worst-case smoothness quantities. For variable step sizes,
this leads to a bound with improved dependence on the
conditioning and without the logarithmic dependence on the
target precision. Under a Lipschitz-continuity assumption
on the Laplacian, it leads to a sharper higher-order
discretization bound, where the usual Hessian-Lipschitz
contribution is replaced by a weaker trace-type
third-order smoothness quantity. Finally, in the finite-sum
setting, this viewpoint yields improved guarantees for
SGLD with fixed point control variates.

The example of generalized linear models illustrates that 
these refinements can lead to substantial gains in 
high-dimensional problems with correlated covariates. 
Specifically, replacing $\Mmax$ by $\Mav$ improves the
Lipschitz-gradient LMC bound by a factor $p$.
Under the Hessian-smoothness assumption, replacing
$M_2\sqrt p$ by $M_\Delta$, together with the replacement
of $\Mmax^{3/2}$ by $\Mmax\sqrt{\Mav}$, improves the
corresponding bound by a factor of order $\sqrt p$. 

Several questions remain open. One natural direction is
to investigate whether similar average-smoothness
phenomena hold for other sampling algorithms, such as
underdamped Langevin schemes \citep{cheng2018a,
dalalyan_riou_2018,zhang2023improved},
proximal or mirror Langevin methods
\citep{durmus2018efficient,brosse2017proximal,
ahn2021efficient,chen22c}, or higher-order and
randomized midpoint discretizations
\citep{shen2019randomized,he2020ergodicity,Mou21,
yu2025parallel}. Another direction is to
develop versions of the present bounds under weaker
convexity assumptions or for more general stochastic
gradient estimators \citep{raginsky2017nonconvex,
cheng20e,ma2015complete}. We leave these questions for
future work.

\section{Postponed Proofs}\label{Sec:proofs}

To streamline the presentation and keep the discussion
focused on the main contributions of the paper, several
technical proofs were deferred from the previous 
sections. They are collected in the present section.

\subsection{Proofs of lemmas used in the error bound of 
the LMC}

We gather in this section the proofs of the lemmas used 
in the proof of \Cref{th:1}. We begin with an auxiliary 
result.

\begin{lemma}\label{lem:1c}
    Let $\bX$ and\/ $\bY$ be two $p$-dimensional 
    square-integrable random vectors. For a Brownian 
    motion $(\bW_t)_{t\geqslant 0}$ independent 
    of\/ $\bY$, we set $\bX_t = \bX+\sqrt{2}\,\bW_t$ 
    and define the process $(\bL_t)_{t\geqslant 0}$ by 
    \begin{align}\label{LD}
        d\bL_t = -\nabla f(\bL_t)\,dt + \sqrt{2}\,
        d\bW_t,\qquad t\geqslant 0;\qquad \bL_0 = \bY. 
    \end{align}
    Let $f$ be $m$-strongly convex and set $\psi(h) = \| 
    \bX_h-\bL_h\|_2^2 + m\int_0^h\|\bX_h-\bL_s\|_2^2\,ds$.
    We have
    \begin{align}
        \psi(h)  &\leqslant 
        \|\bX -\bL_0\|_2^2 + \int_0^h2\big(f(\bX_s) 
        - f(\bL_s)\big)\,ds + R_1(h),\label{eq:7}\\
        \psi(h) 
        &\leqslant \|\bX -\bL_0\|_2^2 + 2h\big(f(\bX_h) 
        - f(\bL_h) \big) + R_2(h),\label{eq:7b}
    \end{align}
    where $R_1(h)$ and $R_2(h)$ are random variables defined by
    \begin{align}
        R_1(h) &= m\int_0^h\big\{\|\bX_h-\bL_s\|_2^2 - 
        \|\bX_s-\bL_s\|_2^2\big\}ds,\\
        R_2(h) &= 2\int_0^h \big\{ f(\bL_h) - f(\bL_s) - 
        \sqrt{2}(\bW_h-\bW_s)^\top\nabla f(\bL_s)\big\}\,ds.
    \end{align}
\end{lemma}

\def\bsg{\boldsymbol{g}}
\begin{proof}
    Recall that if $\bsg:[0,+\infty[\to\mathbb R^p$ is a 
    continuously differentiable function, then $t\mapsto 
    \|\bsg(t)\|_2^2$ 
    is continuously differentiable as well and
    $\frac{d}{dt}\|\bsg(t)\|_2^2 = 2\bsg(t)^\top \bsg'(t)$.
    Hence, for every $h\geqslant 0$,
    \begin{align}
        \|\bsg(h)\|_2^2 = \|\bsg(0)\|_2^2 + 2\int_0^h \bsg(s)^\top \bsg'(s)\,ds.
    \end{align}
    We apply this formula to $\bsg(t) = \bX_t-\bL_t = 
    \bX-\bL_0 + \int_0^t \nabla f(\bL_s)\,ds$. 
    Therefore, $\bsg$ is continuously differentiable and
    $\bsg'(t)=\nabla f(\bL_t)$. We thus get
    \begin{align}\label{eq:aux-1c-1}
        \|\bX_h-\bL_h\|_2^2 & = \|\bX-\bL_0\|_2^2
        + 2\int_0^h (\bX_s-\bL_s)^\top \nabla f(\bL_s)\,ds.
    \end{align}
    Since $f$ is $m$-strongly convex, we have for every $s\in[0,h]$,
    $f(\bX_s) \geqslant f(\bL_s) + (\bX_s - \bL_s)^\top \nabla f(\bL_s)   
    + \frac{m}{2}\|\bX_s-\bL_s\|_2^2$.
    This yields
    \begin{align}\label{eq:lem1:3a}
        2(\bX_s-\bL_s)^\top \nabla f(\bL_s)
        \leqslant 2\big(f(\bX_s)-f(\bL_s)\big)
        - m\|\bX_s-\bL_s\|_2^2.
    \end{align}
    Combining this inequality with \eqref{eq:aux-1c-1} and
    rearranging the terms, we get the first claim of the lemma. 
    
    The second claim is proved as follows. For every $s\in[0,h]$, we decompose $\bX_s-\bL_s= \bX_h-\bL_s - \sqrt{2}\,(\bW_h - 
    \bW_s)$. Hence,
    \begin{align}\label{eq:lem1:2}
        (\bX_s-\bL_s)^\top \nabla f(\bL_s)
        &= (\bX_h-\bL_s)^\top \nabla f(\bL_s)
        - \sqrt{2}\,(\bW_h-\bW_s)^\top \nabla f(\bL_s).
    \end{align}
    Since $f$ is $m$-strongly convex, we have for every $s\in[0,h]$,
    $f(\bX_h) \geqslant f(\bL_s) + (\bX_h - \bL_s)^\top \nabla f(\bL_s)   
    + \frac{m}{2}\|\bX_h-\bL_s\|_2^2$.
    This yields
    \begin{align}\label{eq:lem1:3}
        (\bX_h-\bL_s)^\top \nabla f(\bL_s)
        \leqslant f(\bX_h)-f(\bL_s)
        - \frac{m}{2}\|\bX_h-\bL_s\|_2^2.
    \end{align}
    Combining \eqref{eq:lem1:2} and \eqref{eq:lem1:3}, we 
    obtain
    \begin{align}
        (\bX_s-\bL_s)^\top \nabla f(\bL_s)
        &\leqslant f(\bX_h)-f(\bL_s)
        - \frac{m}{2}\|\bX_h-\bL_s\|_2^2
        - \sqrt{2}\,(\bW_h-\bW_s)^\top \nabla f(\bL_s)\\
        &= f(\bX_h) - f(\bL_h) 
        - \frac{m}{2}\|\bX_h-\bL_s\|_2^2 \\
       &\qquad\qquad+ f(\bL_h) - 
        f(\bL_s)- \sqrt{2}\,(\bW_h-\bW_s)^\top
       \nabla f(\bL_s).
    \end{align}
    Plugging this bound into \eqref{eq:aux-1c-1}, we get
    \begin{align}\label{eq:aux-1c-2}
        \|\bX_h-\bL_h\|_2^2
        \leqslant \|\bX-\bL_0\|_2^2
        &+ 2h \big(f(\bX_h) - f(\bL_h)\big) - m\int_0^h 
        \|\bX_h-\bL_s\|_2^2\,ds \\
        &+ 2\int_0^h\big\{f(\bL_h) - f(\bL_s) - \sqrt{2}
        (\bW_h-\bW_s)^\top \nabla f(\bL_s)\big\}\,ds.
    \end{align}
    Since the integral term in the last line is equal to 
    $R_2(h)$, rearranging the terms yields the second 
    claimed inequality.
\end{proof}

\begin{proof}[Proof of \Cref{lem:2}]
    We take expectations on both sides of 
    \eqref{eq:7b} for a random vector $\bY\sim \pi$. 
    Since for every $s\geqslant 0$, $\bL_s$ has the 
    same distribution as $\bY$, we get 
    $\mathbf E[\|\bX_h-\bL_h\|_2^2]\ge \wass_2^2(\bX_h,\bY)$
    and $\mathbf E[\|\bX_h-\bL_s\|_2^2]\ge \wass_2^2(\bX_h,\bY)$
    with $\bX_h = \bX+\sqrt{2}\,\bW_h$.  
    Furthermore, the expectation of $R_2(h)$ vanishes, and
    we get
    \begin{align}
        (1+mh)\wass_2^2(\bX+\sqrt{2}\,\bW_h,\bY)
        &\leqslant 
        \mathbf E[\|\bX -\bL_0\|_2^2] + 2h\mathbf E\big[f(\bX+\sqrt{2}\,\bW_h) - f(\bY)\big].\label{eq:9}
    \end{align}    
    This coincides with \eqref{cor:1.0}, the first claim 
    of the lemma. 
    
    For the second claim, we proceed in a similar way using 
    \eqref{eq:7} instead of \eqref{eq:7b}: we take the expected 
    value for a random vector $\bY\sim \pi$. This leads to
    \begin{align}
        (1+mh)\wass_2^2(\bX_h,\bY) \leqslant \mathbf E[\|\bX 
        &- \bL_0\|_2^2] + \int_0^h2\mathbf E\big(f(\bX_s) - 
        f(\bL_s)\big)\,ds\\ 
        &+ m\int_0^h \big\{ 
        \mathbf E[\|\bX_h - \bL_s\|_2^2] - \mathbf E[\|
        \bX_s - \bL_s\|_2^2]\big\}\,ds
        .\label{eq:9b}
    \end{align}
    Then, we use the fact that $\bX,\bL_s$ and $\bW_s$ are 
    orthogonal to $\bW_h - \bW_s$ in the sense of $\mathbb 
    L^2$, to obtain
    \begin{align}
        \mathbf{E}[\|\bX_h-\bL_s\|_2^2] &= 
        \mathbf{E}[\|\bX_s + \sqrt{2}\,(\bW_h -
        \bW_s)-\bL_s\|_2^2]
         = \mathbf{E}[\|\bX_s-\bL_s\|_2^2] + 2(h-s)p.
    \end{align}
    Combining this inequality with \eqref{eq:9b}, we get 
    the second claim of the lemma.
\end{proof}

\begin{proof}[Proof of \Cref{lem:basic-gd}]
    On the one hand, we have
    \begin{align}
        \|\btheta-h\nabla f(\btheta)-\by\|_2^2 
        = \|\btheta-\by\|_2^2
        -2h(\btheta-\by)^\top \nabla f(\btheta)
        + h^2\|\nabla f(\btheta)\|_2^2.
        \label{eq:basic-gd-1}
    \end{align}
    On the other hand, since $f$ is $m$-strongly 
    convex, we have $f(\by)\geqslant f(\btheta)
    + (\by-\btheta)^\top \nabla f(\btheta)
    + \frac{m}{2}\|\by-\btheta\|_2^2$. Rearranging
    the terms, we get
    \begin{align}
        -(\btheta-\by)^\top \nabla f(\btheta)
        \leqslant
        f(\by)-f(\btheta)
        - \frac{m}{2}\|\btheta-\by\|_2^2.
        \label{eq:basic-gd-2}
    \end{align}
    Combining \eqref{eq:basic-gd-1} and \eqref{eq:basic-gd-2}, we obtain
    \begin{align}
        \|\btheta-h\nabla f(\btheta)-\by\|_2^2
        &\leqslant
        \|\btheta-\by\|_2^2
        + 2h\big(f(\by)-f(\btheta)\big)
        - mh\|\btheta-\by\|_2^2
        + h^2\|\nabla f(\btheta)\|_2^2 \\
        &= (1-mh)\|\btheta-\by\|_2^2
        + 2h\big(f(\by)-f(\btheta)\big)
        + h^2\|\nabla f(\btheta)\|_2^2.
    \end{align}
    This proves \eqref{eq:basic-gd}.
\end{proof}

\begin{proof}[Proof of \Cref{lem:exp-upper}]
    We first prove \eqref{eq:exp-upper:a}. Let
    $\bzeta_0=\bxi$, and, for $j=1,\ldots,p$, let
    $\bzeta_j$ be the vector obtained from $\bxi$ by
    setting its first $j$ coordinates equal to zero.
    Thus $\bzeta_p=0$ and
    $\bzeta_{j-1}=\bzeta_j+\xi_j\bse_j$. By the definition
    of $M_j$, we have
    \begin{align}
        f(\btheta'+\sqrt{2s}\,\bzeta_{j-1})
        -f(\btheta'+\sqrt{2s}\,\bzeta_j)
        &\leqslant
        \sqrt{2s}\,\xi_j\,
        \partial_j f(\btheta'+\sqrt{2s}\,\bzeta_j)
        +sM_j\xi_j^2.
    \end{align}
    Summing over $j=1,\ldots,p$ gives
    \begin{align}
        f(\btheta'+\sqrt{2s}\,\bxi)-f(\btheta')
        &\leqslant
        \sqrt{2s}\sum_{j=1}^p
        \xi_j\,\partial_j f(\btheta'+\sqrt{2s}\,\bzeta_j)
        +s\sum_{j=1}^p M_j\xi_j^2.
    \end{align}
    Taking expectations and using that $\xi_j$ is 
    independent of $\bzeta_j$, with zero mean and 
    unit variance, yields $\mathbf E\big[ f(\btheta' 
    + \sqrt{2s}\,\bxi)-f(\btheta') \big] \leqslant 
    s\sum_{j=1}^p M_j = \Mav ps$. This proves 
    \eqref{eq:exp-upper:a}.

    We now prove \eqref{eq:exp-upper}. Since the gradient 
    of $f$ is $\Mmax$-Lipschitz, we have
    \begin{align}
        f(\btheta-h\nabla f(\btheta))
        &\leqslant
        f(\btheta)
        -h\|\nabla f(\btheta)\|_2^2
        +\frac{\Mmax h^2}{2}\|\nabla f(\btheta)\|_2^2 \\
        &=
        f(\btheta)
        -h\Big(1-\frac{h\Mmax}{2}\Big)
        \|\nabla f(\btheta)\|_2^2.
    \end{align}
    Multiplying both sides by $2$ gives
    $2\big[f(\btheta-h\nabla f(\btheta))-f(\btheta)\big]
        \leqslant -h(2-h\Mmax)\|\nabla f(\btheta)\|_2^2$.
    This proves \eqref{eq:exp-upper}.
\end{proof}

\subsection{Proof of the error bound for variable step
size stated in \Cref{prop:VSS}}

Let us define the normalized error $
r_k:=\frac{m}{p}\,\wass_2^2(\nu_k^{\sf LMC},\pi)$, 
$k\in\mathbb N$. Since $h_k\leqslant 1/\Mmax$ for 
every $k$, claim \eqref{claim1} of \Cref{th:1} holds
with a non-positive last term. Dropping this term, 
we get
\begin{align}\label{eq:rec-vss}
    r_{k+1}\leqslant
    \frac{1-mh_k}{1+mh_k}\,r_k
    +\frac{(\kappav+1)(mh_k)^2}{1+mh_k},
    \qquad k\in\mathbb N.
\end{align}
We first consider the initial phase $k\leqslant k_0$.
For such $k$, the definition $h_k$ presented in 
\Cref{prop:VSS} simplifies to $h_k=1/\Mmax$, so 
\eqref{claim3} yields
\begin{align}
r_{k_0}\leqslant
 e^{-2k_0/\kappamax}\varepsilon_0^2
 +\frac{\kappav+1}{2\kappamax}.
\end{align}
By the definition of $k_0$, we have
$e^{-2k_0/\kappamax}\varepsilon_0^2
\leqslant \frac{\kappav+1}{2\kappamax}$,
and thus $r_{k_0}\leqslant \frac{\kappav+1}{
\kappamax}$. We now prove by induction that, 
for every $k\geqslant k_0$,
\begin{align}\label{eq:ind-vss}
    r_k\leqslant
    \frac{\kappav+1}{\kappamax+k-k_0}.
\end{align}
The case $k=k_0$ follows directly from
what precedes. Assume next that
\eqref{eq:ind-vss} holds for some $k\geqslant k_0$.
Since $mh_k = {1}/{(\kappamax+k-k_0)}$,
the recursion \eqref{eq:rec-vss} gives
\begin{align}
    r_{k+1}
    &\leqslant
    \frac{\kappamax+k-k_0-1}{\kappamax+k-k_0+1}\,r_k
    +\frac{\kappav+1}
    {(\kappamax+k-k_0)(\kappamax+k-k_0+1)}
    \\
    &\leqslant
    \frac{\kappamax+k-k_0-1}{\kappamax+k-k_0+1}\,
    \frac{\kappav+1}{\kappamax+k-k_0}
    +\frac{\kappav+1}
    {(\kappamax+k-k_0)(\kappamax+k-k_0+1)}
    \\
    &=\frac{\kappav+1}{\kappamax+k-k_0+1}.
\end{align}
This proves \eqref{eq:ind-vss}, which is equivalent to
\eqref{W2VS}.

Finally, if we want
$\wass_2^2(\nu_k^{\sf LMC},\pi)\leqslant \varepsilon^2 p/m$,
it is enough, by \eqref{W2VS}, to ensure that
$\frac{\kappav+1}{\kappamax+k-k_0}\leqslant \varepsilon^2$. 
Solving this inequality for $k$ gives \eqref{kVS}.

\subsection{Proof of the error bound of the LMC when
the Laplacian is Lipschitz continuous}

\begin{proof}[Proof of Theorem~\ref{th:3}]
Let $\bvartheta\sim \nu$ be a random vector in $\mathbb 
R^p$ and $\bY\sim\pi$ be a random vector optimally 
coupled with $\bvartheta$. Let $\bvartheta^+ = 
\bvartheta - h\nabla f(\bvartheta) + \sqrt{2}\,\bW_h$, 
where $\bW$ is a Brownian motion in $\mathbb R^p$ 
independent of $\bvartheta$ and $\bY$. We define the
Langevin process $(\bL_t)_{t\geqslant 0}$ by 
\begin{align}
    d\bL_t = -\nabla f(\bL_t)\,dt + \sqrt{2}\,
    d\bW_t,\qquad\bL_0 = \bY.
\end{align}
Since $\bL_0\sim\pi$, the process $(\bL_t)_{t
\geqslant 0}$ is stationary and has $\pi$ as its 
invariant density, that is $\bL_t\sim \pi$, for every 
$t\geqslant 0$. Setting $\boldsymbol{\Psi}
(\bvartheta,\bL_0) = \bvartheta - h\nabla f
(\bvartheta) - \bL_0 + h\nabla f(\bL_0)$, we get
\begin{align}
    \|\bvartheta^+ - \bL_h\|_{\mathbb L_2}^2
    & = 
    \Big\|\bvartheta - h\nabla f(\bvartheta) - 
    \bL_0 + \int_0^h \nabla f(\bL_t)\,dt
    \Big\|_{\mathbb L_2}^2\\
    &= \|\bvartheta - h\nabla f(\bvartheta) - 
    \bL_0
    +h\nabla f(\bL_0)+ \int_0^h \{\nabla f(\bL_t) 
    - \nabla f(\bL_0)\}\,dt\|_{\mathbb L_2}^2\\
    &=
    \|\boldsymbol{\Psi}(\bvartheta,\bL_0)\|_{
    \mathbb L_2}^2
    +\Big\|\int_0^h \{\nabla f(\bL_t) - \nabla f
    (\bL_0)\}\,dt\Big\|_{\mathbb L_2}^2 + 2\int_0^h
    A_t\,dt\label{eq:4:3}
\end{align}
where
\begin{align}
    A_t & = \mathbf E\Big[\Big(\bvartheta - h\nabla f 
    (\bvartheta) - \bL_0 + h\nabla f(\bL_0)\Big)^\top 
    \Big(\nabla f(\bL_t) - \nabla f(\bL_0)\Big)\Big].
\end{align}
Since $f$ is $m$-strongly convex and has $\Mmax$-Lipschitz
gradient, the strengthened cocoercivity inequality implies,
provided $(m+\Mmax)h\leqslant 2$, that
\begin{align}\label{eq:4.6}
    \big\|\boldsymbol{\Psi}(\bvartheta,\bL_0)
    \big\|_{\mathbb L_2} \leqslant (1-mh)\|\bvartheta 
    - \bL_0\|_{\mathbb L_2}
    =(1-mh)\wass_2(\nu,\pi).
\end{align}
Furthermore, the It\^o formula yields
\begin{align}
    \nabla f(\bL_t) - \nabla f(\bL_0) 
    &=  \int_0^t \nabla^2 f(\bL_s)\,d\bL_s 
    + \int_0^t \Delta \nabla f(\bL_s)\,ds \\
    &= \int_0^t\big\{\nabla \Delta  f - \nabla^2 
    f \nabla f\big\}(\bL_s)\,ds +
    \sqrt{2} \int_0^t \nabla^2 f(\bL_s) \,d\bW_s.
    \label{eq:4:7}
\end{align}
Therefore, combining this decomposition with 
the triangle inequality, and using the stationarity of 
the process $(\bL_t)$, we get
\begin{align}
    \Big\|\int_0^h \{\nabla f(\bL_t) &- \nabla f(\bL_0)\}\,dt
    \Big\|_{\mathbb L_2}  \leqslant \int_0^h \big\|\nabla f
    (\bL_t) - \nabla f(\bL_0)\big\|_{\mathbb L_2}\,dt\\
    & \leqslant \frac{h^2}{2}
    \|\nabla^2 f  \nabla f\|_{\mathbb L_2(\pi)}
    + \frac{2\sqrt{2}}{3}\,h^{3/2} \|\|\nabla^2 f
    \|_{\mathsf F}\|_{\mathbb L_2(\pi)} + \frac{h^2}{2} \|
    \nabla \Delta f\|_{\mathbb L_2(\pi)}\\
    &\leqslant \frac{h^2\Mmax}{2} \|\nabla f\|_{\mathbb
    L_2(\pi)} + \frac{2\sqrt{2}}{3}\,h^{3/2} (\Mmax \|
    \Delta f \|_{\mathbb L_1(\pi)})^{1/2} + \frac{h^2}{2}
    \sqrt{p}\,M_\Delta\\
    &\leqslant \frac{h^2\Mmax}{2} \sqrt{p \Mav} + \frac{2 
    \sqrt{2}}{3}\,h^{3/2} \sqrt{p\Mmax \Mav} + \frac{
    h^2}{2} \sqrt{p}\,M_\Delta\\
    &\leqslant \frac12 h^2  M_\Delta
    \sqrt{p} + \frac{3}{2}\,h^{3/2} \sqrt{\Mmax
    \Mav p}.\label{eq:4:5}
\end{align}
Here, the third line follows from the fact that the 
largest eigenvalue of the Hessian $\nabla^2 f$ at any 
point is bounded by $\Mmax$, the fourth inequality follows 
from \Cref{lem:7} below, and the last inequality is
a consequence of the condition $h\Mmax\leqslant 1$.

\begin{lemma}\label{lem:7}
    Let $f\in C^2(\mathbb R^p)$ be a convex function, and
    let $\Delta f$ denote its Laplacian. Set
    $M_{1,\Delta} = (1/p)\sup_{\btheta} \Delta f(\btheta)$.
    Then $\|\nabla f\|_{\mathbb L_2(\pi)}^2 = \|\Delta f
    \|_{\mathbb L_1(\pi)}\leqslant pM_{1,\Delta} \leqslant
    p\Mav$.
\end{lemma}

\begin{proof}
The integration by
parts formula yields $\|\nabla f\|_{\mathbb L_2(\pi)}^2
     = \int \nabla f^\top \nabla (- \pi) = \int \Delta f\,
     \pi = \|\Delta f\|_{\mathbb L_1(\pi)}$. The rest of 
     the lemma follows from the definitions of $M_{1,
     \Delta}$ and $\Mav$.
\end{proof}
Finally, we bound $A_t$. 
The stochastic integral has zero conditional expectation given
$(\bvartheta,\bL_0)$, and hence its scalar product with
$\boldsymbol{\Psi}(\bvartheta,\bL_0)$
has expectation zero. This leads to
\begin{align}
    A_t &= \mathbf E\Big[\boldsymbol{\Psi}(\bvartheta,
    \bL_0)^\top 
    \Big(\nabla f(\bL_t) - \nabla f(\bL_0) - \sqrt{2}
    \int_0^t \nabla^2 f(\bL_s)\,d\bW_s\Big)\Big]\\
    &\leqslant \Big\|\boldsymbol{\Psi}(\bvartheta,\bL_0)\Big\|_{
    \mathbb L_2}\Big\|\nabla f(\bL_t) - \nabla f(\bL_0) 
    - \sqrt{2}\int_0^t \nabla^2 f(\bL_s)\,d\bW_s
    \Big\|_{\mathbb L_2}\\
    &\leqslant (1-mh)\wass_2(\nu,\pi) \Big(\int_0^t \|\nabla^2 
    f\nabla f\|_{\mathbb L_2(\pi)} + \|\nabla \Delta f\|_{
    \mathbb L_2(\pi)}\,ds\Big)\\
    &\leqslant t(1-mh) \wass_2(\nu,\pi) (\Mmax \sqrt{\Mav} + 
    M_\Delta)\sqrt{p}.\label{eq:4:8}
\end{align}
Here, in the second line, we used the Cauchy-Schwarz 
inequality, in the third line we used inequality 
\eqref{eq:4.6} and decomposition \eqref{eq:4:7}, whereas
in the last line we used the definitions of $\Mmax$ and
$M_\Delta$ combined with \Cref{lem:7}. 

We have now found upper bounds on all the three terms
appearing in the right-hand side of \eqref{eq:4:3}. 
Indeed, combining \eqref{eq:4.6}, \eqref{eq:4:5} and 
\eqref{eq:4:8}, we arrive at
\begin{align}
    \|\bvartheta^+ - \bL_h\|_{\mathbb L_2}^2 
    \leqslant (1-mh)^2\wass_2^2(\nu,\pi) &+ \Big(\frac32 h^{3/2} 
    \sqrt{\Mmax \Mav} + \frac{h^2}{2}M_\Delta\Big)^2p \\
    &
    + h^2(1-mh)\wass_2(\nu,\pi) (\Mmax \sqrt{\Mav} + 
    M_\Delta)\sqrt{p}.
\end{align}
Since $\bL_h\sim\pi$, the left-hand side controls
$\wass_2^2(\nu^+,\pi)$. 
Finally, using \cite[Lemma 7]{dalalyan2019user} 
with $A = mh$, $C = \tfrac12 h^2(\Mmax \sqrt{\Mav}  + 
M_\Delta) \sqrt{p}$ and $B =   \frac12 h^2  M_\Delta
\sqrt{p} + \frac{3}{2}\,h^{3/2} \sqrt{\Mmax
\Mav p}$,
we get
\begin{align}
    \wass_2(\nu_k^{\textsf{LMC}},\pi) &\leqslant 
    (1-mh)^k \wass_2(\nu_0,\pi) + \frac{(2\Mmax 
    \sqrt{\Mav}  + M_\Delta)h\sqrt{p}}{m}.
\end{align}
This completes the proof of the theorem. 
\end{proof}

\subsection{Proof of the error bound of the SGLD}

We prove \Cref{th:sgld} in several steps. Since some 
of these intermediate estimates may be of independent 
interest, we collect them in the following proposition. 
Its last claim is precisely the assertion of 
\Cref{th:sgld}.

\begin{proposition}\label{prop:th:sgld}
    Let $f:\mathbb R^p\to\mathbb R$ be $n\mbar$-strongly 
    convex with $n\Mmaxbar$-Lipschitz gradient, and 
    assume that the average coordinate-wise smoothness 
    constant of $f$ is $n\Mavbar$. Assume also that
    $f=\sum_{i=1}^n f_i$, with $L_i = \|\nabla f_i\|_{
    \rm Lip}<\infty$. Let $\bvartheta$ be an $\mathbb 
    R^p$-valued random vector, let $\bxi\sim\mathcal N(
    0,\mathbf I_p)$, and let $I\sim\mathrm{Unif}\{1, 
    \ldots,n\}$ be independent.
    Define $\nu$ and $\nu^+$ as the laws of $\bvartheta$
    and $\bvartheta^+  = \bvartheta - hn\big(\nabla 
    f_I(\bvartheta) - \nabla f_I(\btheta_*)\big) + 
    \sqrt{2h}\,\bxi$, respectively. Then, for every 
    $h\geqslant 0$, 
    \begin{align}
        \wass_2^2(\nu^+,\pi)
        &\leqslant
        \frac{1-n\mbar h}{1+n\mbar h}\wass_2^2(\nu,\pi)
        + \frac{n(\Mavbar+\mbar)h^2p}{1+n\mbar h}      
        + \frac{h^2(hn\Mmaxbar-2)}{1+n\mbar h}
        \|\nabla f(\bvartheta)\|_{\mathbb L_2}^2        \\
        &\qquad
        + nh^2\sum_{i=1}^n
        \big\|\nabla f_i(\bvartheta)
        -\nabla f_i(\btheta_*)\big\|_{\mathbb L_2}^2 .
        \label{claim-sgld-1}
    \end{align}
    If, in addition, $hn\Mmaxbar\leqslant 2$, then
    \begin{align}
        \wass_2^2(\nu^+,\pi)
        &\leqslant
        \Big(
        \frac{1-n\mbar h}{1+n\mbar h}
        +2(n\Mtwo h)^2\Big)\,
        \wass_2^2(\nu,\pi)    
        + \Big(\frac{\Mavbar+\mbar}{1+n\mbar h}
        + \frac{2\Mtwo^2}{\mbar}\Big)nh^2p.
        \label{claim-sgld-2}
    \end{align}
    Consequently, if SGLD is run with constant step-size
    $h$ satisfying $h \leqslant \frac{\mbar}{3n\Mtwo^2}$,
    then
    \begin{align}
        \wass_2^2(\nu_k^{\sf SGLD},\pi)
        \leqslant
        (1+n\mbar h)^{-k}\,
        \wass_2^2(\nu_0^{\sf SGLD},\pi)
        + \Big(\Mavbar+2\mbar
        +\frac{2\Mtwo^2}{\mbar}\Big)\frac{hp}{\mbar}.
        \label{claim-sgld-3}
    \end{align}
\end{proposition}

\begin{proof}
The proof follows the same lines as that of
\Cref{th:1}. We have of course $m = n\mbar$, $\Mmax 
= n\Mmaxbar$ and $\Mav = n\Mavbar$. Let us define
\begin{align}
    \widehat{\nabla f}_I(\btheta)
    :=
    n\big(\nabla f_I(\btheta)-\nabla f_I(\btheta_*)\big).
\end{align}
Since $\nabla f(\btheta_*)=0$, we have
\begin{align}
    \mathbf E[\widehat{\nabla f}_I(\btheta)]
    &=
    n\bigg(\frac1n\sum_{i=1}^n \nabla f_i(\btheta)
    - \frac1n\sum_{i=1}^n \nabla f_i(\btheta_*)\bigg)
    =
    \nabla f(\btheta).
\end{align}
Let $\bar\bvartheta^+ = \bvartheta-h\nabla f(\bvartheta) 
+ \sqrt{2h}\,\bxi$ be the conditional expectation of
$\bvartheta^+$ with respect to $I$, given $(\bvartheta, 
\bxi)$. Let $\bY\sim\pi$ be optimally coupled with
$\bar\bvartheta^+$.
We choose this coupling independently of $I$, conditionally
on $(\bvartheta,\bxi)$. Then
\begin{align}
    \wass_2^2(\nu^+,\pi)
    &\leqslant
    \|\bvartheta^+-\bY\|_{\mathbb L_2}^2 =
    \|\bar\bvartheta^+-\bY\|_{\mathbb L_2}^2
    + h^2
    \|\widehat{\nabla f}_I(\bvartheta)
    -\nabla f(\bvartheta)\|_{\mathbb L_2}^2,
\end{align}
because the cross term has zero expectation. On the one
hand,
\begin{align}
    \|\widehat{\nabla f}_I(\bvartheta)-\nabla f
    (\bvartheta)\|_{\mathbb L_2}^2
    =
    \|\widehat{\nabla f}_I(\bvartheta)\|_{\mathbb L_2}^2
    -
    \|\nabla f(\bvartheta)\|_{\mathbb L_2}^2.
\end{align}
On the other hand, applying \eqref{eq:2:9}  gives
\begin{align}
    \|\bar\bvartheta^+-\bY\|_{\mathbb L_2}^2
    &\leqslant
    \frac{1- m h}{1+ m h}\,
    \wass_2^2(\nu,\pi)
    + \frac{(\Mav+ m)h^2p}{1+ m h}  
    + \frac{h^2(h\Mmax-1)}{1+m h}
    \|\nabla f(\bvartheta)\|_{\mathbb L_2}^2.
\end{align}
Combining the last three displays, we obtain
\begin{align}
    \wass_2^2(\nu^+,\pi)
    &\leqslant
    \frac{1-m h}{1+m h}\,
    \wass_2^2(\nu,\pi)
    + \frac{(\Mav+m)h^2p}{1+m  h} 
    + \frac{h^2(h\Mmax-2)}{1+ m h}
    \|\nabla f(\bvartheta)\|_{\mathbb L_2}^2\\
    &\qquad\qquad\qquad\qquad + h^2
    \|\widehat{\nabla f}_I(\bvartheta)\|_{\mathbb L_2}^2.
    \label{eq:3:7}
\end{align}
This proves \eqref{claim-sgld-1}, since
$\|\widehat{\nabla f}_I(\bvartheta)\|_{\mathbb L_2}^2 =
n\sum_{i=1}^n \|\nabla f_i(\bvartheta) - \nabla 
f_i(\btheta_*)\|_{\mathbb L_2}^2$.

Assume now that $hn\Mmaxbar\leqslant 2$. Then the third
term on the right-hand side of \eqref{eq:3:7} is
nonpositive and can be dropped. Since each $f_i$ has
$L_i$-Lipschitz gradient, we have $\|\nabla f_i 
(\bvartheta)-\nabla f_i(\btheta_*)\|_{\mathbb L_2}
\leqslant L_i\|\bvartheta-\btheta_*\|_{\mathbb L_2}$. 
Therefore,
\begin{align}
    \|\widehat{\nabla f}_I(\bvartheta)\|_{\mathbb L_2}^2
    &= n\sum_{i=1}^n
    \|\nabla f_i(\bvartheta)
    - \nabla f_i(\btheta_*)\|_{\mathbb L_2}^2 \leqslant
    n^2\Mtwo^2
    \|\bvartheta-\btheta_*\|_{\mathbb L_2}^2.
\end{align}
Let $\bY\sim\pi$ be optimally coupled with $\bvartheta$.
Then
\begin{align}
    \|\bvartheta-\btheta_*\|_{\mathbb L_2}^2
    &\leqslant
    2\|\bvartheta-\bY\|_{\mathbb L_2}^2
    +2\|\bY-\btheta_*\|_{\mathbb L_2}^2 =
    2\wass_2^2(\nu,\pi)
    +2\|\bY-\btheta_*\|_{\mathbb L_2}^2.
\end{align}
Since $\bY\sim\pi\propto e^{-f}$ and
$\nabla f(\btheta_*)=0$, integration by parts gives
$\mathbf E\big[(\bY-\btheta_*)^\top\nabla f(\bY)
\big] = p$. By strong convexity with parameter $m$,
\begin{align}
    m\|\bY-\btheta_*\|_2^2  &\leqslant
    (\bY-\btheta_*)^\top \big(\nabla f(\bY)-\nabla f(\btheta_*)\big) =  (\bY-\btheta_*)^\top 
    \nabla f(\bY).
\end{align}
Taking expectations yields $\|\bY-\btheta_* 
\|_{\mathbb L_2}^2 \leqslant \frac{p}{m}$.
Consequently,
\begin{align}
    nh^2\sum_{i=1}^n
    \|\nabla f_i(\bvartheta)
    -\nabla f_i(\btheta_*)\|_{\mathbb L_2}^2
    \leqslant
    2n^2h^2\Mtwo^2\wass_2^2(\nu,\pi)
    +\frac{2n^2h^2\Mtwo^2p}{m}.
\end{align}
Substituting this into \eqref{claim-sgld-1} gives
\eqref{claim-sgld-2}.

Finally, assume that $ h\leqslant \frac{\mbar}{3n
\Mtwo^2}$. Since $\mbar\leqslant \Mmaxbar\leqslant
\Mtwo$, we have
\begin{align}
    \frac{1-m h}{1+m h}
    +2n^2\Mtwo^2h^2 \leqslant
    \frac{1}{1+n\mbar h},\quad\text{and}\quad
    \frac{2\Mtwo^2}{\mbar}
    \leqslant
    \frac{1}{1+n\mbar h}
    \left(\mbar+\frac{2\Mtwo^2}{\mbar}\right).
\end{align}
Hence \eqref{claim-sgld-2} implies
\begin{align}
    \wass_2^2(\nu_{k+1}^{\sf SGLD},\pi)
    &\leqslant
    \frac{1}{1+n\mbar h}\,
    \wass_2^2(\nu_k^{\sf SGLD},\pi)  
    + \frac{n(\Mavbar+2\mbar)h^2p}{1+n\mbar h}
    + \frac{2n\Mtwo^2h^2p}
    {\mbar(1+n\mbar h)}.
\end{align}
Iterating this recursion, we get
\begin{align}
    \wass_2^2(\nu_k^{\sf SGLD},\pi)
    &\leqslant
    \frac{\wass_2^2(\nu_0^{\sf SGLD},\pi)}
    {(1+n\mbar h)^k} 
    + \frac{1-(1+n\mbar h)^{-k}}{n\mbar h}
    \Big(\Mavbar+2\mbar+\frac{2\Mtwo^2}{\mbar}\Big)
    nh^2p.
\end{align}
Since $1-(1+n\mbar h)^{-k}\leqslant 1$, this yields
\begin{align}
    \wass_2^2(\nu_k^{\sf SGLD},\pi)
    \leqslant
    (1+n\mbar h)^{-k}
    \wass_2^2(\nu_0^{\sf SGLD},\pi)
    + \Big\{\Mavbar+2\mbar+\frac{2\Mtwo^2}{\mbar}\Big\}
    \frac{hp}{\mbar}.
\end{align}
This proves \eqref{claim-sgld-3}.
\end{proof}

\begin{proof}[Proof of \Cref{th:sgld}]
The condition $h\leqslant \mbar/(3n\Mtwo^2)$ implies
$h\leqslant 2/(n\Mmaxbar)$, since
$\mbar\leqslant \Mmaxbar\leqslant \Mtwo$. Therefore
\eqref{claim-sgld-3} applies and gives precisely the
claim of \Cref{th:sgld}.
\end{proof}
\subsection{Evaluation of Lipschitz constants in the 
generalized linear models}

\begin{proof}[Proof of \Cref{prop:glm-constants}]
Throughout the proof, $C$ denotes a positive constant
depending only on $C_g$, whose value may change from 
line to line. We write $\bfS = \sum_{i=1}^n \bX_i
\bX_i^\top$. We shall use the following standard 
consequences of Gaussian concentration and covering 
arguments. With probability at least $1/2$, the 
following bounds hold simultaneously:
\begin{align}
    \sum_{i=1}^n\|\bX_i\|_2^4
    &\leqslant  Cn\,\{{\rm tr}(\bSigma)\}^2,
    \qquad
    \sup_{\|\bu\|_2=1}
    \sum_{i=1}^n(\bX_i^\top\bu)^4
    \leqslant  Cn\|\bSigma\|_{\rm op}^2 .
    \label{eq:glm-event-4}
\end{align}
The Cauchy-Schwarz inequality implies that 
\eqref{eq:glm-event-4} yields
\begin{align}\label{eq:glm-event-1}
    \|\bfS\|_{\rm op}
    &\leqslant  Cn\|\bSigma\|_{\rm op},
    \qquad\qquad\qquad
    \sum_{i=1}^n\|\bX_i\|_2^2
    \leqslant  Cn\,{\rm tr}(\bSigma).
\end{align}
In the sequel, we work on the event on which all
the inequalities above hold.

For every $\btheta\in\mathbb R^p$ and for every 
$j \in\{1,\ldots,p\}$, we have
\begin{align}\label{eq:f''}
    \nabla^2 f(\btheta) = \sum_{i=1}^n
    g_i''(\bX_i^\top\btheta)\bX_i\bX_i^\top 
    \quad\text{and}\quad 
    \partial_{jj}^2 f(\btheta) = \sum_{i=1}^n
    g_i''(\bX_i^\top\btheta)X_{ij}^2. 
\end{align}
Since $|g_i''|\leqslant  C_g$, this 
implies
\begin{align}
    \Mmax \leqslant  C\|\bfS\|_{\rm op}
    \leqslant  Cn\|\bSigma\|_{\rm op}.
    \label{eq:glm-Mmax}
\end{align}
This gives $\Mmax\leqslant Cn$ when $\bSigma=\mathbf I_p$,
and $\Mmax\leqslant Cnp$ when
$\bSigma=\mathbf I_p+\mathbf 1_p\mathbf 1_p^\top$.

Let us now consider $\Mav$. In view of \eqref{eq:f''}, 
for every coordinate $j$, $ M_j\leqslant  C\sum_{i=1}^n 
X_{ij}^2$.  Averaging over $j$ and using the inequality 
${\rm tr}(\bSigma)\leqslant 2p$ gives
\begin{align}
    \Mav = \frac1p\sum_{j=1}^p M_j \leq
    \frac{C}{p}\sum_{i=1}^n\|\bX_i\|_2^2
    \leqslant Cn\,\frac{{\rm tr}(\bSigma)}{p}
    \leqslant Cn.    
    \label{eq:glm-Mav}
\end{align}
Next, the Lipschitz seminorm of $\nabla f_i$ 
defined by $f_i(\btheta)=g_i(\bX_i^\top\btheta)$ 
satisfies
\begin{align}
    \|\nabla f_i\|_{\textrm{Lip}}\leqslant 
    \sup_{\btheta\in\mathbb R^p}\|\nabla^2 f_i
    (\btheta)\|_{\textrm{op}} \leqslant  C\|
    \bX_i\|_2^2 .
\end{align}
Combining this inequality with $\tr(\bSigma)
\leqslant 2p$, we arrive at
\begin{align}
    \Mtwo^2 = \frac1n\sum_{i=1}^n L_i^2
    \leq \frac{C}{n}\sum_{i=1}^n\|\bX_i\|_2^4 
    \leq C \tr(\bSigma)^2\leqslant Cp^2.
    \label{eq:glm-Mtwo}
\end{align}
We turn to $M_\Delta$. Fix a $\btheta\in\mathbb R^p$. 
Define the $p\times n$ matrix $\bfX = [\bX_1,\ldots 
\bX_n]$ and the $n$-vector $\boldsymbol{a}$ with entries 
$a_i = g_i'''(\bX_i^\top\btheta)\|\bX_i\|_2^2$. 
We have 
\begin{align}
    \Delta f(\btheta) = \sum_{i=1}^n g_i''(\bX_i^\top
    \btheta)\|\bX_i\|_2^2 \quad\text{and}\quad
    \nabla\Delta f(\btheta) = \sum_{i=1}^n g_i'''
    (\bX_i^\top\btheta)\|\bX_i\|_2^2\bX_i = 
    \bfX\boldsymbol{a}.
\end{align}
Thus,
\begin{align}
    \|\nabla\Delta f(\btheta)\|_2^2
    & = \|\bfX\boldsymbol{a}\|_2^2\leqslant \|\bfX^\top
    \bfX\|_{\text{op}}\|\boldsymbol{a}\|_2^2 = \|\bfX
    \bfX^\top\|_{\text{op}}\|\boldsymbol{a}\|_2^2
    \leqslant C_g^2\Big(\sum_{i=1}^n\|\bX_i\|_2^4\Big)
    \|\bfS\|_{\rm op}.
\end{align}
Using \eqref{eq:glm-event-1} and \eqref{eq:glm-event-4},
we get
\begin{align}
    \sup_{\btheta\in\mathbb R^p}
    \|\nabla\Delta f(\btheta)\|_2 \leq Cn\,{\rm tr}
    (\bSigma)\|\bSigma\|_{\rm op}^{1/2} \leqslant
    Cnp\,\|\bSigma\|_{\rm op}^{1/2}.
\end{align}
In view of the definition of $M_\Delta$, 
we may take
\begin{align}
    M_\Delta = \frac1{\sqrt{p}} \sup_{\btheta\in
    \mathbb R^p} \|\nabla\Delta f(\btheta)\|_2
    \leq
    Cn\,{\sqrt p}\,
    \|\bSigma\|_{\rm op}^{1/2}.
    \label{eq:glm-Mdelta}
\end{align}
For $\bSigma=\mathbf I_p$, this gives
$M_\Delta\leqslant Cn\sqrt p$. For
$\bSigma=\mathbf I_p+\mathbf 1_p\mathbf 1_p^\top$,
we have $\|\bSigma\|_{\rm op}=p+1\leqslant 2p$ and
${\rm tr}(\bSigma)=2p$, and therefore
$M_\Delta\leqslant Cnp$.

It remains to bound $M_2$. For
$\bh=\btheta'-\btheta$, the mean-value theorem yields
\begin{align}
    \nabla^2 f(\btheta')-\nabla^2 f(\btheta)
    =
    \sum_{i=1}^n
    a_i(\btheta,\btheta')(\bX_i^\top\bh)
    \bX_i\bX_i^\top,
\end{align}
where $|a_i(\btheta,\btheta')|\leqslant  C_g$. Hence
\begin{align}
    \|\nabla^2 f(\btheta')-\nabla^2 f(\btheta)\|_{\rm 
    op}^2 &\leqslant C 
    \sup_{\|\bu\|_2=1} \Big\{\sum_{i=1}^n |\bX_i^\top\bh|
    (\bX_i^\top\bu)^2\Big\}^2\\
    &\leqslant C\sum_{i=1}^n(\bX_i^\top\bh)^2 
    \sup_{\|\bu\|_2=1} \sum_{i=1}^n (\bX_i^\top\bu)^4.
\end{align}
Using \eqref{eq:glm-event-1} and \eqref{eq:glm-event-4},
we obtain
\begin{align}
    M_2
    \leq
    Cn\|\bSigma\|_{\rm op}^{3/2}.
    \label{eq:glm-M2}
\end{align}
Therefore, if $\bSigma=\mathbf I_p$, then
$M_2\sqrt p\leqslant Cn\sqrt p$. If
$\bSigma=\mathbf I_p+\mathbf 1_p\mathbf 1_p^\top$,
then $\|\bSigma\|_{\rm op}\leqslant 2 p$, and hence
$M_2\sqrt p\leqslant Cnp^2$.

Combining \eqref{eq:glm-Mmax}, \eqref{eq:glm-Mav},
\eqref{eq:glm-Mtwo}, \eqref{eq:glm-Mdelta} and
\eqref{eq:glm-M2} gives all the entries of the table.
\end{proof}


\section*{Acknowledgements} 
The work was supported by ERC grant SAGMOS (grant 
agreement No. 101201229).

\bibliographystyle{imsart-number}
\bibliography{bibliography}

\end{document}